\documentclass{amsart}
\usepackage[english]{babel}
\usepackage{amsmath,amssymb,dsfont,mathrsfs,anysize,fancyhdr,epsfig}
\usepackage{cite}
\usepackage{graphicx}
\usepackage{url}
\usepackage{hyperref}
\hypersetup{colorlinks,
	linkcolor=blue,%
	citecolor=green}

\usepackage[capitalise]{cleveref}

\usepackage{tikz}
\usepackage{enumerate}

\usepackage{xcolor}

\usepackage[utf8]{inputenc}
\usepackage{chngcntr}
\usepackage{apptools}
\AtAppendix{\counterwithin{theorem}{section}}

\newcommand{\rr}{\mathbf{r}}

\renewcommand{\ss}{\mathbf{s}}
\renewcommand{\ll}{\mathbf{l}}

\newcommand{\vv}{\mathbf{v}}
\newcommand{\ww}{\mathbf{w}}

\newcommand{\zz}{\mathbf{z}}

\DeclareMathOperator{\Id}{Id}

\DeclareMathOperator{\Aut}{Aut}
\DeclareMathOperator{\dom}{dom}

\DeclareMathOperator{\ess}{ess}

\DeclareMathOperator{\Diam}{Diam}
\DeclareMathOperator{\Diag}{Diag}
\DeclareMathOperator{\dist}{dist}
\DeclareMathOperator{\Conv}{Conv}

\DeclareMathOperator{\Mat}{Mat}
\DeclareMathOperator{\Cstar}{C^*}

\DeclareMathOperator{\Image}{Image}

\newcommand{\CC}{\mathbb{C}}

\newcommand{\NN}{\mathbb{N}}

\newcommand{\RR}{\mathbb{R}}
\newcommand{\TT}{\mathbb{T}}

\newcommand{\ZZ}{\mathbb{Z}}

\newcommand{\calA}{\mathcal{A}}
\newcommand{\calB}{\mathcal{B}}

\newcommand{\calH}{\mathcal{H}}

\newcommand{\calM}{\mathcal{M}}
\newcommand{\calN}{\mathcal{N}}
\newcommand{\calO}{\mathcal{O}}
\newcommand{\calP}{\mathcal{P}}

\newcommand{\calS}{\mathcal{S}}

\newcommand{\calU}{\mathcal{U}}

\newcommand{\calW}{\mathcal{W}}
\newcommand{\calX}{\mathcal{X}}

\newtheorem{theorem}{Theorem}[section]
\newtheorem{lemma}[theorem]{Lemma}

\newtheorem{definition}[theorem]{Definition}
\newtheorem{proposition}[theorem]{Proposition}

\theoremstyle{remark}

\newtheorem{remark}[theorem]{Remark}

\title[Commutator estimates for normal operators in factors and applications to derivations]{Commutator estimates for normal operators in factors with applications to derivations}

\begin{document}
\author[A. Ber]{Aleksei F. Ber}
\address{
	Department of Mathematics \\National University of Uzbekistan\\
	Vuzgorodok, 100174\\ Tashkent, Uzbekistan\\
	\emph{E-mail~:} {\tt ber@ucd.uz}
}
\author[M.J. Borst]{Matthijs J. Borst}
\address{Delft Institute of Applied Mathematics \\Delft University of Technology\\
	P.O. Box 5031,
	2600 GA\\ Delft, The Netherlands\\
	\emph{E-mail~:} {\tt  m.j.borst@tudelft.nl }
}
\author[F. Sukochev]{Fedor A. Sukochev}
\address{School of Mathematics and Statistics, University of New South Wales, Kensington, 2052, NSW, Australia
	\emph{E-mail~:} {\tt f.sukochev@unsw.edu.au}
}
\date{\today}

	\begin{abstract}
	For a normal measurable operator $a$ affiliated with a von Neumann factor $\calM$ we show:
			
				If $\calM$ is infinite, then there is $\lambda_0\in \CC$ so that for $\varepsilon>0$ there are $u_{\varepsilon}=u_{\varepsilon}^*$, $v_{\varepsilon}\in \calU(\calM)$ with
				$$v_\varepsilon|[a,u_\varepsilon]|v_\varepsilon^*\geq(1-\varepsilon)(|a-\lambda_0\textbf{1}|+u_\varepsilon|a-\lambda_0\textbf{1}|u_\varepsilon).$$
				
				If $\calM$ is finite, then there is $\lambda_0\in\CC$ and $u,v\in\calU(\calM)$ so that
				$$v|[a,u]|v^*\geq \frac{\sqrt{3}}{2}(|a-\lambda_0\textbf{1}|+u|a-\lambda_0\textbf{1}|u^*).$$
		
			These bounds are optimal for infinite factors, II$_1$-factors and some I$_n$-factors. Furthermore, for finite factors applying $\|\cdot\|_{1}$-norms to the inequality provides estimates on the norm of the inner derivation $\delta_{a}:\calM\to L_1(\calM,\tau)$ associated to $a$. While  by \cite[Theorem 1.1]{BBS} it is known for finite factors and self-adjoint $a\in L_1(\calM,\tau)$ that $\|\delta_{a}\|_{\calM\to L_1(\calM,\tau)} = 2\min_{z\in \CC}\|a-z\|_{1}$, we present concrete examples of finite factors $\calM$ and normal operators $a\in \calM$ for which  this fails.
	\end{abstract}
\maketitle

	\section{Introduction}
	Derivations are linear maps $\delta$ that satisfy the Leibniz rule $\delta(xy) = \delta(x)y+x\delta(y)$. They play an essential role in the theory of Lie algebras, Cohomology, the study of Semi-groups and in Quantum Physics, see \cite{HSsurvey,KL,SS}.
	A classical result on derivations is due to Stampfli \cite{Stampfli} which asserts that for $a\in B(H)$, a bounded operator on a Hilbert space $H$, the derivation $\delta_{a}:B(H)\to B(H)$ defined by the commutator $\delta_a(x) = [a,x]=ax-xa$ has operator norm $\|\delta_{a}\| = 2\inf_{z\in \CC}\|a-z\mathbf{1}\|$. Through the work of \cite{KLR,Gaj,Zsido}, the result of Stampfli has been extended to derivations on arbitrary von Neumann algebras $\calM$ (see also \cite{Magajna} for more in this direction). More precisely, the result of Zsid\'{o} \cite[Corollary]{Zsido} asserts that for $\calM$ a von Neumann algebra and $a\in \calM$, the derivation $\delta_a:\calM\to \calM$  associated to $a$ satisfies the distance formula:
	\begin{align}\label{eq:derivation-distance-formula}
		\|\delta_{a}\|_{\calM\to\calM} = 2\min_{z\in Z(\calM)}\|a-z\|,
	\end{align}
	where $Z(\calM)$ denotes the center of $\calM$. 
	
	Our research aims to obtain results similar to \eqref{eq:derivation-distance-formula} for derivations that map $\calM$ into the predual $\calM_*$.
	Indeed, the predual $\calM_*$ is a $\calM$-bimodule (see \cref{section:estimates-for-derivations}) and therefore it is possible to consider derivations $\delta:\calM\to \calM_*$. 
	Important work on such derivations was done in \cite{BP,Haagerup,BGM} and particularly the result of \cite[Theorem 4.1]{Haagerup} showed that all these derivations are inner (i.e. of the form $\delta=\delta_{a}$ for some $a\in \calM_*$, defined by $\delta_a(x) = ax-xa$). These studies arose after Connes proved in \cite{Connes} that  all amenable $\Cstar$-algebras are necessarily nuclear. Haagerup proved in  \cite{Haagerup} that the reverse implication is also true. 
	
	In \cite{BBS} the norms of these derivations were studied and  results analogouos to \eqref{eq:derivation-distance-formula} were found in certain cases: for $\calM$ properly infinite it was shown that some form of formula \eqref{eq:derivation-distance-formula} holds true and for $\calM$ finite the same was proved under the condition that $a$ is self-adjoint. The proofs of these results were based on improvements of the operator estimates obtained in \cite{BS2012, BSWA}, see below:
	\begin{theorem}\label{bs_t1}\cite[Theorem 1]{BS2012}
		Let $\calM$ be a factor and let $a=a^*\in S(\calM)$ (here $S(\calM)$ is the algebra of measurable operators attached to $\calM$).
		\begin{enumerate}
			\item If $\calM$ is a finite factor or else a purely infinite $\sigma$-finite factor, then there
			exists $\lambda_0\in\RR$ and $u_0=u_0^*\in\calU(\calM)$, such that
			$$|[a,u_0]|=u_0|a-\lambda_0\textbf{1}|u_0+|a-\lambda_0\textbf{1}|$$
			where $\calU(\calM)$ is the group of all unitary operators in $\calM$;
			\item there exists $\lambda_0\in\RR$ so that for any $\varepsilon>0$ there exists $u_\varepsilon=u_\varepsilon^*\in\calU(\calM)$
			such that
			\begin{align}\label{bs_t1_2}
				|[a,u_\varepsilon]|\geq (1-\varepsilon)|a-\lambda_{0}\textbf{1}|.
			\end{align}
		\end{enumerate}
	\end{theorem}
    This theorem was extended to arbitrary von Neumann algebras in \cite{BSWA} with the replacement of $\lambda_{0}\textbf{1}$ by an element from the center. In \cite[Theorem B.1]{BBS} inequality \eqref{bs_t1_2} was extended to:
    \begin{align}\label{bs_t1_2_1}
    |[a,u_\varepsilon]|\geq (1-\varepsilon)(|a-\lambda_{0}\textbf{1}|+u_\varepsilon|a-\lambda_{0}\textbf{1}|u_\varepsilon).
    \end{align}

    The question arises: is such an inequality as \eqref{bs_t1_2_1} true for arbitrary $a\in S(\calM)$? More precisely, are there such $\lambda_0\in\CC,\ u,v,w\in \calU(\calM)$ and a constant $C>0$ such that
    \begin{align}\label{bs_t1_2_2}
    |[a,u]|\geq C(v|a-\lambda_{0}\textbf{1}|v^*+w|a-\lambda_{0}\textbf{1}|w^*)
    \end{align}
	holds true?
    In this paper, we give an answer to this question in the case when $a$ is a normal operator (see Theorems \ref{t3}, \ref{t_inf}). It turns out that if $\calM$ is an infinite factor, then the constant $C$ can be chosen arbitrarily close to $1$, just as in the case of self-adjoint $a$. However,  in the case when $\calM$ is a finite factor, the situation changes. For II$_1$-factors the optimal constant $C$ turns out to be equal to $\frac{\sqrt{3}}{2}$ and for I$_n$-factors appropriate upper and lower bounds on the optimal constant are given by $\Lambda_n \leq C\leq \frac{1}{2}\widetilde{\Lambda}_n$ (see \eqref{eq:definition_lambda_n} and \eqref{eq:definition-tilde-lambda_n} for definitions of these constants and \eqref{eq:estimates-lambda} for estimates).
    We summarize above results in the following theorem.
    \begin{theorem}[see Theorems \ref{t3}, \ref{t_inf}]\label{theorem_main}
    	Let $\calM$ be a factor and let $a\in S(\calM)$ be normal. Then there is a $\lambda_0\in \CC$ and unitaries $u,v,w\in \calU(\calM)$ such that
    	\begin{align}
    		|[a,u]| \geq C\left(v|a-\lambda_0\mathbf{1}|v^* + w|a-\lambda_0\mathbf{1}|w^*\right)
    	\end{align}
    	for some constant $C>0$ independent of $a$. Moreover
    	\begin{enumerate}
    		\item when $\calM$ is a I$_n$-factor, $n<\infty$, the optimal constant satisfies $\Lambda_n \leq C\leq \frac{1}{2}\widetilde{\Lambda}_n$.
    		\item when $\calM$ is a II$_1$-factor, the optimal constant is $C=\frac{\sqrt{3}}{2}$.
    		\item when $\calM$ is an infinite factor, we can choose $C$ arbitrarily close to $1$.
    	\end{enumerate}
    \end{theorem}
	This theorem can be applied to obtain norm estimates for derivations $\delta:\calM\to \calM_*$ and extend results of \cite{BBS}. Specifically, we consider the case that $\calM$ is finite, and $\tau$ is a faithful normal tracial state on $\calM$. In this case $\calM_*$ is isomorphic to $L_1(\calM,\tau)$ (see e.g. \cite[Lemma 2.12 and Theorem 2.13]{Tak2}). 
	As an application of inequality \eqref{bs_t1_2_1}, it was proved in \cite[Theorem 1.1]{BBS} that, for  $a=a^*\in L_1(\calM,\tau)$, we have
	\begin{align}\label{bbs2_1}
		\left\| \delta_a \right\|_{\calM\to L_1(\calM,\tau)} = 2 \min_{z\in Z(S(\calM))} \left\| a-z \right\|_1
	\end{align} (here $Z(S(\calM))$ denotes the center of $S(\calM)$) and that the minimum is attained at a self-adjoint element  $c_a=c_a^* \in L_1(\calM,\tau)\cap Z(S(\calM))$.
In the present paper, using \cref{theorem_main}, we show that for a finite factor $\calM$ and for an arbitrary normal measurable $a\in L_1(\calM,\tau)$, the estimate
	\begin{align}\label{eq:intro:derivation-estimate}
	\sqrt{3}\min_{z\in \CC}\left\| a-z \right\|_1\leq \left\| \delta_a \right\|_{\calM\to L_1(\calM,\tau)}\leq 2\min_{z\in \CC}\left\| a-z \right\|_1
	\end{align} holds (see  \cref{tderivationbound}). In \cref{section:estimates-for-derivations} we show that the estimates given in \eqref{eq:intro:derivation-estimate} are sharp. In particular, in \cref{tderivationbound} we demonstrate that for any finite II$_1$-factor $\calM$ there exists a normal $a\in \calM$ such that the derivation $\delta_{a}$ is non-zero and satisfies $\|\delta_{a}\|_{\calM\to L_1(\calM,\tau)}=\sqrt{3}\min_{z\in \CC}\|a-z\|_1$, whereas it follows from \cref{t_inf} and \cite[Theorem 3.1]{BBS} that for any infinite factor $\calM$ formula \eqref{bbs2_1} holds for an arbitrary normal $a\in L_1(\calM,\tau)$.\\	
	
	Finally, we remark that \eqref{eq:intro:derivation-estimate} is in fact an estimate for the $L_1$-diameter of the unitary orbit $\calO(a) = \{uau^*: u\in \calU(\calM)\}$ of $a$ as $\Diam_{L_1(\calM,\tau)}(\calO(a)) = \|\delta_{a}\|_{\calM\to L_1(\calM,\tau)}$, see end of Section \ref{section:estimates-for-derivations}.

   \subsection{Structure and overview}
   In \cref{section:preliminaries} we introduce standard terminology, recall the definitions of (locally) measurable operators and prove \cref{p_AAP_21} and \cref{t_AAP_21} that extend some results to locally measurable operators. In \cref{section:designations} we introduce the constants $\Lambda_n$ and $\widetilde{\Lambda}_n$ for $n\in \NN\cup \{\infty\}$ that will be used throughout the paper. In \cref{section:technical-theorems} our main result is \cref{theorem-transformation-with-bound}, which is closely related to the constants $\Lambda_n$ and to the operator inequality \eqref{bs_t1_2_2}. In \cref{section:estimates-in-finite-factors} we use this result to obtain \cref{t3} which establishes the operator inequality of \cref{theorem_main} for normal elements in finite factors. In \cref{section:estimates-in-infinite-factors} we obtain the inequality of \cref{theorem_main} for normal locally measurable operators affiliated with an infinite factor, see \cref{t_inf}.
   In \cref{section:estimates-for-derivations} we apply our results to obtain the estimate \eqref{eq:intro:derivation-estimate} for the norm of derivations $\delta_a:\calM\to L_1(\calM,\tau)$ for normal $a\in L_1(\calM,\tau)$, and we show the given bounds are optimal in some cases. 
    In the Appendix we prove two technical results regarding the constants $\Lambda_n$ and $\widetilde{\Lambda}_n$. In particular, \cref{t_techmain} determines the exact value of $\Lambda_n$ for $n\not=4$.

	\section{Preliminaries}\label{section:preliminaries}
	We establish notation on von Neumann algebras and (locally) measurable operators (for a thorough discussion of these topics we refer to \cite{DdPS,FK}). Furthermore, we prove two results, \cref{p_AAP_21} and \cref{t_AAP_21}, which generalize a known result (a type of triangle inequality for operators) to locally measurable operators. 
	
    Let $\calM$ be a von Neumann algebra on a Hilbert space $H$ with unit $\textbf{1}$. We let $\calU(\calM)$ be the group of unitaries in $\calM$, let $\calP(\calM)$ be the lattice of projections in  $\calM$ and let  $Z(\calM)$ be the center of $\calM$. 
    
	Recall that two projections $e, f \in  \calM$ are called \textit{Murray-von Neumann equivalent} (denoted by $e\sim f$)  if there exists an element
    $u \in \calM$ such that $u^\ast  u = e$ and $u u^\ast  = f.$
    A projection $p \in \calM$ is called \textit{finite}, if the
    conditions $q \leq  p$ and $q\sim p$  imply that $q = p.$

	Let $x : \dom(x) \to  H$ be a densely defined closed linear operator 
    (the domain $\dom(x)$ of $x$ is a linear subspace in $H$). Then $x$ is said to be \textit{affiliated} with $\calM$
    if $yx \subset  xy$ for all $y$ from the commutant $\calM'$  of the algebra $\calM.$
    A linear operator $x$ affiliated with $\calM$ is called \textit{measurable} with respect to $\calM$ if
    $\chi_{(\lambda,\infty)}(|x|)$ is a finite projection for some $\lambda>0.$ Here
    $\chi_{(\lambda,\infty)}(|x|)$ is the  spectral projection of $|x|$ corresponding to the interval $(\lambda, +\infty).$
    We denote the set of all measurable operators by $S(\calM).$
    Clearly, $\calM$  is a subset of $S(\calM).$
    It is clear that if $\calM$ is a factor of type $I$ or $III$ then $S(\calM)=\calM$.

    Let $x, y \in  S(\calM).$ It is well known that $x+y$ and
    $xy$ are densely-defined and preclosed
    operators \cite{DdPS}. We define the \textit{strong sum} respectively the \textit{strong product} of $x$ and $y$ as the closures of these operators, which we simply also denote by $x+y$ and $xy$ respectively. When $S(\calM)$ is equipped with the operation of strong sum, operation of strong product, and the $*$-operation, it becomes a unital $*$-algebra over $\CC$.
    It
    is clear that $\calM$  is a $\ast$-subalgebra of $S(\calM)$. Moreover, in the case that $\calM$ is finite, every operator affiliated with $\calM$ becomes measurable.
    In particular, the set of all affiliated operators then forms a $\ast$-algebra, which coincides with
    $S(\calM).$
    Following \cite{KL, KLT}, in the case when the von Neumann algebra $\calM$ is finite, we refer to the algebra $S(\calM)$ as the Murray-von Neumann algebra associated with $\calM$.

    Let $\calM$ be semi-finite and let $\tau$ be a faithful normal semi-finite trace on $\calM.$  A linear operator $x$ affiliated with $\calM$ is called \textit{$\tau$-measurable} with respect to $\calM$ if
    $\tau(\chi_{(\lambda,\infty)}(|x|))<\infty$ for some $\lambda>0.$ We denote the set of all $\tau$-measurable operators by $S(\calM,\tau).$ The set $S(\calM,\tau)$ is a $*$-subalgebra of $S(\calM)$ that contains $\calM$.
    Consider the topology  $t_{\tau}$ of convergence in measure or \textit{measure topology}
    on $S(\calM,\tau),$ which is defined by
    the following neighborhoods of zero:
    $$
    N(\varepsilon, \delta)=\{x\in S(\calM,\tau): \exists \, e\in \calP(\calM), \, \tau(\mathbf{1}-e)\leq\delta, \, xe\in
    \calM, \, \|xe\|_\calM\leq\varepsilon\},
    $$
    where $\varepsilon, \delta$
    are positive numbers.  The algebra $S(\calM,\tau)$ equipped with the measure topology is a topological $*$-algebra and $F$-space \cite{DdPS}.

    A linear operator $x$ affiliated with $\calM$ is called \textit{locally measurable} with respect to $\calM$ if there exist increasing central projections $(p_n)$ in $\calP(Z(\calM))$ converging strongly to $\textbf{1}$, and such that $xp_n\in S(\calM)$. The set $LS(\calM)$ of locally measurable operators forms a $*$-algebra with respect to the operations of a strong sum and a strong product.
    It is clear that if $\calM$ is a factor then $LS(\calM)=S(\calM)$.

    Let $x\in LS(\calM)$. Denote by $\ll(x)$ - the left carrier of $x$, by $\rr(x)$ - the right carrier of $x$ and $\ss(x)=\ll(x)\vee \rr(x)$. If $x=u|x|$ is the polar decomposition of $x$, then $\ll(x)=uu^*$ and $\rr(x)=u^*u$.
    We denote $\Re(x) = \frac{x+x^*}{2}$ and $\Im(x) = \frac{x-x^*}{2i}$ for respectively the real and imaginary part of $x$.
    For a self-adjoint $x\in LS(\calM)$ we denote by $x_+$ (respectively, $x_-$) its positive (respectively negative) part, defined by $x_+=\frac{x+|x|}{2}$ (respectively, $x_-=-\frac{x-|x|}{2}$). We note that $x_-$ and $x_+$ are orthogonal, that is $x_-x_+=0$.

   	We require \cref{t_AAP_21} which states a triangle inequality for operators $x\in LS(\calM)$. The statement is similar to \cite[Theorem 2.2]{AAP1982} where for operators $x\in \calM$ the result was shown with partial isometries instead of isometries (see also \cite[Lemma 4.3]{FK} and \cite[Lemma 4.15]{Hiai}).
    To prove \cref{t_AAP_21}, we will need the following statement which is similar to \cite[Proposition 2.1]{AAP1982}. 
    Here, $v\in \calM$ is called an \textit{isometry} if $v^*v=\mathbf{1}$.
    \begin{proposition}\label{p_AAP_21}
    	For each $x\in LS(\calM)$ there is an isometry $v\in\calM$ such that $\Re(x)_+\leq v|x|v^*$.
    \end{proposition}
    \begin{proof}
    	Let $p=\ss(\Re(x)_+),\ a=p(x+|x|)$. Then clearly $\ll(a)\leq p$. We show $p=\ll(a)$. Put $r=p-\ll(a)$ so that
    	$0=ra = rar=rxr+r|x|r$. Taking the real part of this equation gives $0=r\Re(x)r+r|x|r$.
    	Since $r\leq p$ we have $r\Re(x)_-r=0$ and therefore $r\Re(x)r=r\Re(x)_+r$. Then
    	$0=r\Re(x)r+r|x|r=r\Re(x)_+r+r|x|r$
    	and hence $r\Re(x)_+r=0$. Then as $(\Re(x)_+^{\frac{1}{2}}r)^*(\Re(x)_+^{\frac{1}{2}}r) = r\Re(x)r=0$, we obtain $\Re(x)_+^{\frac{1}{2}}r = 0$ and hence $\Re(x)_+r=0$. Therefore, $\Re(x)_+(\mathbf{1}-r) = \Re(x)_+$ which shows $(\mathbf{1}-r)\geq \ss(\Re(x)_+)=p$ and we conclude $r=0$, i.e. $p=\ll(a)$.
    	
    	Let $a=w|a|$ be the polar decomposition of $a$. Then $ww^*=p$. Put $q=w^*w$ and $s= (\mathbf{1}-q)\wedge p$. We show $s=0$. Indeed
    	$as=aqs=0$, thus $s(x+|x|)s=sas=0$ and taking the real part of this equation gives $s\Re(x)s+s|x|s=0$. As $s\leq p$ we have $s\Re(x)_-s=0$ so that $s\Re(x)s = s\Re(x)_+s$.
    	Again, by the same arguments as before, this implies $s\Re(x)_+s=0$ and subsequently $(\mathbf{1}-s)\geq p$. Thus $s\leq (\mathbf{1}-p)\wedge p=0$.
    	
    	Let $(\mathbf{1}-p)(\mathbf{1}-q)=w_0|(\mathbf{1}-p)(\mathbf{1}-q)|$ be the polar decomposition of $(\mathbf{1}-p)(\mathbf{1}-q)$. Then $w_0w_0^*\leq\mathbf{1}-p$ and $w_0^*w_0\leq\mathbf{1}-q$. Moreover, if $t=\mathbf{1}-q-w_0^*w_0=\mathbf{1}-q-\rr((\mathbf{1}-p)(\mathbf{1}-q))$ then we see $(\textbf{1}-q)t = t$ and
    	$$(\mathbf{1}-p)t=((\mathbf{1}-p)(\mathbf{1}-q))t=0\Rightarrow t\leq p\Rightarrow t\leq s=0.$$
    	Thus we obtain the equality $w_0^*w_0 = \textbf{1}-q$ and obtain that $v=w+w_0$ is an isometry in $\calM$.
    	
    	The inequality  $\Re(x)_+\leq v|x|v^*$ is proved in the same way as in the proof of \cite[Proposition 2.1]{AAP1982} (the monotonicity of the square root function follows from \cite[Corollary 2.2.28]{DdPS}).
    \end{proof}
	The proof of \cref{t_AAP_21} is exactly the same as the proof of \cite[Theorem 2.2]{AAP1982}, but instead of \cite[Proposition 2.1]{AAP1982} we use \cref{p_AAP_21} above. We include the proof for completeness.
    \begin{theorem}\label{t_AAP_21}
    	For any $x,y\in LS(\calM)$ there are isometries $v,w\in\calM$ such that
    	$$|x+y|\leq v|x|v^*+w|y|w^*.$$
    \end{theorem}
    \begin{proof}
     We write the polar decomposition $x+y = u|x+y|$. Then 
    	\begin{align}
    		|x+y| = \frac{1}{2}(u^*(x+y) + (x+y)^*u) = \Re(u^*x) + \Re(u^*y)
    	\end{align}
    Furthermore, $|u^*x| = (x^*u^*ux)^\frac{1}{2} \leq \|u\|(x^*x)^{\frac{1}{2}} \leq |x|$ and similarly $|u^*y|\leq |y|$. Now apply \cref{p_AAP_21} to $u^*x$ and to $u^*y$ to obtain isometries $v,w\in \calM$ so that
    \begin{align}
    	|x+y| = \Re(u^*x) + \Re(u^*y)\leq v|u^*x|v^* + w|u^*y|w^* \leq v|x|v^* + w|y|w^*
    \end{align}
    
    \end{proof}

	\section{{Constants $\Lambda_n$ and $\widetilde{\Lambda}_n$}}\label{section:designations}
    For $n\in\NN$ we denote by $(\Omega_n,\mu_n)$ the set $\{1,2,\ldots,n\}$ equipped with the normalized counting measure, and by $(\Omega_\infty,\mu_\infty)$ we denote the interval $[0,1]$ equipped with Lebesgue measure. We will moreover write $S(\Omega_n)$ for the set of {complex measurable functions on $\Omega_n$, which is simply the collection of all $n$-tuples of complex numbers}. We write $\Aut_n$  for the automorphism group of $(\Omega_n,\mu_n)$, $n\in\mathbb{N}\cup\{\infty\}$, where automorphism is defined as follows:
	\begin{definition}
		Let $(X_1,\mu_1)$ and $(X_2,\mu_2)$ be measure spaces. We will say that a map $T$ is an isomorphism between $X_1$ and $X_2$ if $T$ is a measurable bijective map $T:N_1\to N_2$ between two sets $N_1\subseteq X_1$ and $N_2\subseteq X_2$ of full measure, and such that moreover $T^{-1}$ is also measurable, and $\mu_1\circ T^{-1} = \mu_2$. Whenever $(X_1,\mu_1)=(X_2,\mu_2)$ we will call $T$ an automorphism.
	\end{definition}
	Let $n\in \NN\cup \{\infty\}$. We now introduce two constant $\Lambda_n$ and $\widetilde{\Lambda}_n$ as follows. Let $g\in S(\Omega_n)$, $T\in \Aut_n$, $z\in\CC$, and put $$\Lambda(g,T,z)=\ess\inf\dfrac{|g-g\circ T|}{|g-z|+|g\circ T-z|},$$
where we assume $\frac{0}{0}=1$. By the triangle inequality we have $|g-g\circ T|\leq |g-z| + |g\circ T-z|$ which shows $\Lambda(g,T,z)\leq 1$ for all $g,T,z$.
We put $$\Lambda(g)=\sup\{\Lambda(g,T,z):\ T\in \Aut_n, z\in\CC\}$$
and define $\Lambda_n$ by 
\begin{align}\label{eq:definition_lambda_n}
	\Lambda_n=\inf_{g\in S(\Omega_n)}\Lambda(g).
\end{align}
For $n>1$ we define $\widetilde{\Lambda}_n$ by setting
	\begin{align}\label{eq:definition-tilde-lambda_n}
	\widetilde{\Lambda}_n = \begin{cases}
		2 & \text{if} \ n=2,\ n=4 \\
		\sqrt{3} & \text{if}\ n=3k,\\
		\frac{2\sqrt{3}}{\sqrt{\frac{3k-3}{3k+1}} + \frac{3k+3}{3k+1}} & \text{if}\ n=3k+1, \ n\not=4\\
		\frac{2\sqrt{3}}{\sqrt{\frac{3k+6}{3k+2}} + \frac{3k}{3k+2}} & \text{if}\ n=3k+2,\\
		\sqrt{3} & \text{if}\ n=\infty.
	\end{cases}.
\end{align}
In the Appendix we will prove two results on the constants $\Lambda_n$ and $\widetilde{\Lambda}_n$.
In \cref{t_techmain} we will precisely determine $\Lambda_n$ for all values except for $n=4$. It turns out that 
\begin{align}\label{eq:estimates-lambda}
	\Lambda_1 = \Lambda_2 = 1, \quad \text{ and }\quad \frac{\sqrt{3}}{2}\leq \Lambda_4\leq 1, \quad \text{ and }\quad  \Lambda_n = \frac{\sqrt{3}}{2} \text{ for } n\not\in \{1,2,4\}.
\end{align}
We observe that this implies that $2\Lambda_n\leq \widetilde{\Lambda}_n$ for $n>1$ with equality when $n\equiv 0 \mod 3$ or $n=\infty$  and that moreover $\lim_{n\to\infty}2\Lambda_n = \sqrt{3} = \lim\limits_{n\to\infty}\widetilde{\Lambda}_n$.

We denote the diameter of a set $A\subseteq \CC$ by $\Diam(A) := \sup_{z,w\in A}|z-w|$.
In  \cref{lemma:optimality-triangle-function} we will show for $n>1$  that there exists $g\in L_{\infty}(\Omega_n)$ with $\Diam(g(\Omega_n))=1$ and $\widetilde{\Lambda}_n = \sup_{z\in \CC}\frac{1}{\|g-z\|_1}$, which will be used throughout the text.

	\section{Technical result} \label{section:technical-theorems}
This section is devoted to the proof of \cref{theorem-transformation-with-bound}, which is closely connected to the operator inequality \eqref{bs_t1_2_2} and to the constants $\Lambda_n$.
To fully state the result we first give the following definition:
\begin{definition}\label{d_con}
	Let $z\in\CC,\ 0\leq\alpha\leq\pi$. The sets $A,B\subset\CC$ will be called \textbf{$(z,\alpha)$-conjugate} if there are two lines in $\CC$ that intersect at the point $z$ at an angle $\alpha$, such that the sets $A$ and $B$ lie in opposite closed corners with the vertex $z$ and the magnitude $\alpha$ (see  \cref{fig:conjugate-sets})
\end{definition} \
	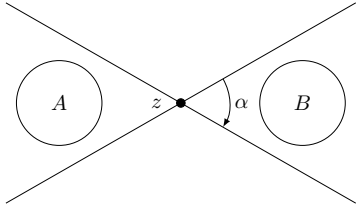
\begin{figure}[h!]
	\scalebox{0.8}{
		\begin{tikzpicture}
			\node at (-0.4,0) 	  	(z0) {$z$};
			\node at (-3,2*0.866)  		(L1A) {};
			\node at (3,-2*0.866)   	(L1B) {};
			\node at (-3,-2*0.866)   		(L2A) {};
			\node at (3,2*0.866)  	 	(L2B) {};
			\node at (1,0) 	  	(alpha) {$\alpha$};

			\draw[-] (L1A) -- node[xshift =-0.2cm, yshift=-0.2cm] {} (L1B);
			\draw[-] (L2A) -- node[xshift =0.2cm, yshift=-0.2cm] {} (L2B);

			\filldraw (0,0) circle (2pt);
			\draw (2,0) circle (20pt);
			\node at (2,0) (B) {$B$};
			
			\draw (-2,0) circle (20pt);
			\node at (-2,0) (A) {$A$};
			
			\draw[-latex] (30:0.8) arc  (30:-30:0.8) ;
			
	\end{tikzpicture}}
	\caption{Two $(z,\alpha)$-conjugate sets $A$ and $B$ are depicted.}
	\label{fig:conjugate-sets}
\end{figure}

\begin{remark}\label{r1}
	Let the sets $A,B$ be  $(z,\alpha)$-conjugate, $a\in A,\ b\in B$. It is easy to see that $$|a-b|\geq (|a-z|+|b-z|)\cos\frac{\alpha}{2}.$$
	Indeed, it is enough to consider the projections of points $a,b$ on the bisector of the angle $\alpha$.
\end{remark}

\begin{theorem}\label{theorem-transformation-with-bound}
	Let $g\in S(\Omega_n),\ n\in\NN\cup\{\infty\}$. Then there exists a $z_0\in \CC$ and an automorphism $T$ of $\Omega_n$ such that
	\begin{align}\label{e1}
		|g\circ T - g|\geq  \frac{\sqrt{3}}{2}\big(|g - z_0| + |g\circ T -z_0|\big).
	\end{align}
 i.e.
\begin{align}\label{e2}
	\Lambda(g)\geq\frac{\sqrt{3}}{2}.
\end{align}
Moreover, the set $\Omega_n$ can be partitioned  into disjoint measurable sets as follows:
\begin{enumerate}[(i)]
	\item if $n$ is even  or $n=\infty$ then  there is a partition $\{X_1\}\cup \{X_2^{m,i}: 1\leq m, 1\leq i\leq 2\}$ so that $g(X_1)\subset\{z_0\},\ \mu_n(X_2^{m,1})=\mu_n(X_2^{m,2})$ and the sets $g(X_2^{m,1}),\ g(X_2^{m,2})$ are $(z_0,\frac{\pi}{3})$-conjugate for $m=1,2,\dots$; Moreover, denoting $X_2 = \Omega_n\setminus X_1$ we have that  $T^k|_{X_k}=\Id_{X_k}$ for $k=1,2$. 
	\item if $n$ is odd then  there is a partition $X_1,X_2,X_3,X_5$, so that $T^k|_{X_k}=Id_{X_k},\ k=1,2,3,5$.
\end{enumerate}
If $n<\infty$ then there exists $z_0\in \CC$ and $T\in Aut_n$ so that
\begin{align}\label{e3}
	\Lambda(g,T,z_0)=\Lambda(g).
\end{align}
\end{theorem}

The above theorem relates to the operator inequality \eqref{bs_t1_2_2} through functional calculus. This is best visible in the case of finite-dimensional factors, see \cref{I_n}.
Furthermore, we note that \cref{theorem-transformation-with-bound} provides a lower bound on the constants $\Lambda_n$. Indeed, given $g\in S(\Omega_n)$ the obtained $z_0$, $T$ are such that $\Lambda(g,T,z_0)\geq \frac{\sqrt{3}}{2}$. Hence $\Lambda_n\geq \frac{\sqrt{3}}{2}$ for all $n\in \NN\cup \{\infty\}$. In the Appendix, \cref{t_techmain}, it is proved that in fact $\Lambda_n = \frac{\sqrt{3}}{2}$ for $n=3$ and $n\geq 5$. This means that, for these values of $n$, the constant $\frac{\sqrt{3}}{2}$ in the above theorem is best possible (i.e. maximal so that for all $g\in S(\Omega_n)$ there exist $z_0, T$ satisfying \eqref{e1}).

The proof of \cref{theorem-transformation-with-bound} is somewhat technical and requires two lemmas, 
\cref{l1} and \cref{l2} 
We give a sketch of the proof. 
   Given a measurable function 
   $g:\Omega_n\to \CC$ we first use \cref{l2} to locate a point $z_0\in\CC$, and divide the plane into $6$ components by drawing $3$ lines intersecting in $z_0$ making angles of $\frac{2\pi}{6}$. The way we do this is such that the measure of the inverse image of $g$ of opposing components is equal. We can then construct an automorphism $T$ by just mapping the inverse image of 
   $g$ of each component to the inverse image of its opposing component. For all $\omega\in \Omega$, we then obtain the estimate
   $\angle g(w),z_0,g(T(w))\geq \frac{2\pi}{3}$ for the angle. \cref{l1} will then imply that \eqref{e1} holds true.
	In the actual proof of \cref{theorem-transformation-with-bound} some difficulties arise with the boundaries of the components, and particularly for the case that we are dealing with the measure space $\Omega_n$ with $n$ odd. Because of this reason, it is necessary to consider multiple cases in the proof.\\

The following lemma gives for complex numbers $z_0,z_1,z_2$ a sufficient condition for 
\begin{align}\label{eq:ellipse}
	|z_1-z_2|\geq \frac{\sqrt{3}}{2}(|z_1-z_0| + |z_2-z_0|)
\end{align}
to hold, namely when the angle satisfies $\angle z_1 z_0 z_2\geq \frac{2\pi}{3}$. Equation \eqref{eq:ellipse} can also be described geometrically as saying that the point $z_0$ lies in the ellipse with foci $z_1$ and $z_2$ and eccentricity $\frac{\sqrt{3}}{2}$.	

	\begin{lemma}\label{l1}
		Let $z_0,z_1,z_2\in \CC$ be points in the plane, and consider the triangle $\triangle z_0z_1z_2$. Denote $a = |z_1-z_2|$, $b = |z_1 - z_0|$, $c = |z_2 - z_0|$, and $\alpha = \angle z_1z_0z_2$. If $\alpha\geq\frac{2\pi}{3}$ then $$a\geq\frac{\sqrt{3}}{2}(b+c).$$
	\end{lemma}
	
\begin{proof}
			According to the cosine theorem we have
			$$a^2=b^2+c^2-2bc\cos\alpha.$$
			Since $\cos\alpha\leq -\frac{1}{2}$ and $b^2+c^2\geq 2bc$ we obtain
			$$4a^2\geq 4(b^2+c^2 + bc) \geq 3b^2 + 3c^2 + 6bc=3(b+c)^2$$
			which shows the result.
\end{proof}

The following lemma is used, for a given function $g\in S(\Omega_n)$, to choose the point $z_0\in \CC$ adequately such that \eqref{e1} holds for some automorphism $T$ that we will later determine. The point $z_0\in \CC$ should be thought of as the center (or rather \textit{a} center) of the image of $g$. In \cref{l2} we have identified $\CC$ with $\RR^2$ and the point $z_0\in \CC$ is represented as a vector $\zz_0\in \RR^2$. This vector $\zz_0$ is chosen together with three affine hyperplanes (i.e. lines) through $\zz_0$ that are represented by unit vectors $\vv_1,\vv_2,\vv_3$ orthogonal to those affine hyperplanes. The unit vectors $\vv_1,\vv_2,\vv_3$ moreover make angles $\angle \vv_i 0\vv_{j}$ for $i\not=j$ of $\frac{2\pi}{3}$ (this means that the affine hyperplances intersect at angles of $\frac{2\pi}{6}$). To each of the affine hyperplanes correspond two closed halfspaces. The lemma tells us that $\zz_0,\vv_1,\vv_2,\vv_3$ can be chosen in such a way that the inverse image of $g$ of each of these closed halfspaces has measure larger or equal to $\frac{1}{2}$. This explains why we think of $\zz_0$ as a center of the image of $g$. Namely, for all three affine hyperplanes it must hold that an equal portion of the domain is mapped to each side (or possibly on the affine hyperplane).
However, we remark that such a `center point' $\zz_0$ with the above properties does not need to be unique.

	\begin{lemma}\label{l2}
		Let $(\Omega,\mu)$ be a probability space and let $g$ be a measurable $\RR^2$-valued function. Then, there exists a point $\zz_0\in \RR^2$, unit vectors $\vv_1,\vv_2,\vv_3\in \RR^2$ with angles $\angle \vv_10\vv_2 = \angle \vv_20\vv_3 = \angle \vv_30\vv_1 = \frac{2\pi}{3}$ so that for $i=1,2,3$, denoting $a_i := \langle \zz_0,\vv_i\rangle$,  we have
		\begin{align*}
			m_i^L:= \mu\bigg(\{\omega\in \Omega:\ \langle g(\omega), \vv_i\rangle \leq a_i\}\bigg)\geq \frac{1}{2}, & &			
			m_i^R:=\mu\bigg(\{\omega\in \Omega:\ \langle g(\omega), \vv_i\rangle \geq a_i\}\bigg)\geq \frac{1}{2}.
		\end{align*}
	For $i=1,2,3$ we point out that $m_i^L + m_i^R=1$ holds if and only if $\mu(\{\omega\in \Omega: \langle g(w),\vv_i\rangle =a_i\})=0$.
    \end{lemma}
	\begin{proof}
	We first prove the result for the case that $g$ is bounded. Denote $\TT=\RR/2\pi\ZZ$ and for $t\in \TT$ set $\vv(t) = (\cos(t),\sin(t))$ and define
    $$\Omega(t,r)=\{\omega\in \Omega: \langle g(\omega), \vv(t)\rangle \leq r\},\ r\in\RR,$$
    $$A(t)=\biggl\{r\in \RR: \frac{1}{2}\leq \mu(\Omega(t,r))\biggr\},$$
	$$a(t) = \inf A(t).$$
    If $r_n\downarrow a(t)$ and $\frac{1}{2}\leq \mu(\Omega(t,r_n))$ then $\Omega(t,r_1)\supset\Omega(t,r_2)\supset\dots$ and $\Omega(t,a(t))=\bigcap_n \Omega(t,r_n)$. Hence,
    \begin{align}\label{eq:12at}
    \frac{1}{2}\leq \mu(\Omega(t,a(t))).
    \end{align}
    If $r_n\uparrow a(t)$ then $\frac{1}{2}\geq \mu(\Omega(t,r_n))$ and $\Omega(t,r_1)\subset\Omega(t,r_2)\subset\dots$ and
    $\{\omega\in \Omega: \langle g(\omega), \vv(t)\rangle <a(t)\}=\bigcup_n \Omega(t,r_n)$. Hence,
	\begin{align}\label{eq:infimum-properties-measure}
    \mu\bigg(\{\omega\in \Omega: \langle g(\omega), \vv(t)\rangle <a(t)\}\bigg)\leq \frac{1}{2}\leq \mu\bigg(\{\omega\in \Omega: \langle g(\omega), \vv(t)\rangle \leq a(t)\}\bigg)
    \end{align}
    and therefore
    \begin{align}\label{eq:-12at}
    \mu\bigg(\{\omega\in \Omega: \langle g(\omega), \vv(t)\rangle\geq a(t)\}\bigg)\geq \frac{1}{2}.
    \end{align}

	We note that it follows from the definition of $a$ that
	\begin{align}\label{eq:pi-shifted-expression-supremum}
	a(t+\pi) = -\sup\biggl\{r\in \RR: \frac{1}{2}\leq \mu(\{\omega\in \Omega: \langle g(\omega), \vv(t)\rangle \geq r\})\biggr\}
	\end{align}
    since
    $$\Omega(t+\pi,r)=\{\omega\in \Omega: \langle g(\omega), \vv(t)\rangle \geq -r\},\ r\in\RR.$$

	Hence, we obtain by \eqref{eq:-12at}, \eqref{eq:pi-shifted-expression-supremum} and by properties of the supremum that $a(t)\leq -a(t+\pi)$ for all $t\in \TT$ since $a(t)\in\biggl\{r\in \RR: \frac{1}{2}\leq \mu(\{\omega\in \Omega: \langle g(\omega), \vv(t)\rangle \geq r\})\biggr\}$.
Moreover, in the second inequality of \eqref{eq:infimum-properties-measure}, replacing $t$ by $t+\pi$  we obtain
\begin{align}\label{eq:pi-shifted-measure-inequality}
	\frac{1}{2}\leq \mu\bigg(\{\omega\in \Omega: \langle g(\omega), \vv(t)\rangle \geq -a(t+\pi)\}\bigg).
\end{align}
Hence, for any $t\in \TT$, and any $b\in [a(t),-a(t+\pi)]$ we obtain using \eqref{eq:infimum-properties-measure} and \eqref{eq:pi-shifted-measure-inequality} that
	\begin{align}\label{eq:lmaineq}
	\frac{1}{2}&\leq \mu\bigg(\{\omega\in \Omega: \langle g(\omega), \vv(t)\rangle \leq b\}\bigg), &			\frac{1}{2} &\leq
	\mu\bigg(\{\omega\in \Omega: \langle g(\omega), \vv(t)\rangle \geq b\}\bigg).
    \end{align}

		We show that the function $a$ is continuous.
	  Indeed, let $\varepsilon>0$, and choose $\delta>0$ such that $\|\vv(t)-\vv(s)\|_2<\varepsilon$ for all $t,s\in \TT$ with $\dist(s,t)<\delta$. Now, fix $t,s\in \TT$ with $\dist(t,s)<\delta$. Then for $\omega\in \Omega$ we have $$|\langle g(\omega),\vv(t)\rangle-\langle g(\omega),\vv(s)\rangle|\leq \|g\|_{\infty}\|\vv(t)-\vv(s)\|_2<\varepsilon\|g\|_{\infty}.$$ But this means for $r\in \RR$ that
	\begin{align*}
	 \{\omega\in \Omega:\ \langle g(\omega),\vv(t)\rangle\leq r \}\subseteq \{\omega\in \Omega:\ \langle g(\omega),\vv(s)\rangle\leq r+\varepsilon\|g\|_\infty \}.
	\end{align*}
	This implies in particular that
	\begin{align*}\frac{1}{2}\leq \mu\bigg(\{\omega:\ \langle g(\omega),\vv(t)\rangle\leq a(t) \}\bigg)\leq\mu\bigg(\{\omega:\ \langle g(\omega),\vv(s)\rangle\leq a(t)+\varepsilon\|g\|_\infty \}\bigg)
	\end{align*} so that $a(s)\leq a(t) + \varepsilon\|g\|_\infty$.  By symmetry of $s$ and $t$ we obtain similarly $a(t)\leq a(s)+\varepsilon\|g\|_{\infty}$, which implies $|a(t)-a(s)|<\varepsilon\|g\|_\infty$ and shows the continuity of $a$.
	
	Now, for $t\in \TT$ and $b\in \RR$ consider the line
	\[L(t,b) = \{\ww\in \RR^2:\ \langle \ww,\vv(t)\rangle = b\} = b\vv(t) + \RR\vv(t+\frac{\pi}{2}).\]
	For $s\not=t \mod \pi$, the lines $L(s)$ and $L(t)$ intersect at a unique point $\ww(L(s,b),L(t,c))$. In particular there is a $r \in \RR$ such that
	\begin{align*}
		\ww(L(s,b),L(t,c)) =  b\vv(s)+r\vv(s+\frac{\pi}{2}).
	\end{align*}
	Therefore
	$c :=\langle 	\ww(L(s,b),L(t,c)),\vv(t)\rangle = b\langle \vv(s),\vv(t)\rangle + r\langle \vv(s+\frac{\pi}{2}),\vv(t)\rangle$ so that
	$r = \frac{c - b\langle \vv(s),\vv(t)\rangle}{\langle \vv(s+\frac{\pi}{2}),\vv(t)\rangle}$ and thus \[\ww(L(s,b),L(t,c))= b\vv(s) + \frac{c -b\langle \vv(s),\vv(t)\rangle}{\langle \vv(s+\frac{\pi}{2}),\vv(t)\rangle}\vv(s+\frac{\pi}{2}).\]
	Let $t\in \TT$. We are interested in finding values $b_1,b_2,b_3\in \RR$ such that the lines
	$L(t-\frac{2\pi}{3},b_1)$, $L(t+\frac{2\pi}{3},b_2)$ and $L(t,b_3)$ intersect at a single point. This  is to say that the intersection point $\ww(L(t-\frac{2\pi}{3},b_1),L(t+\frac{2\pi}{3},b_2))$ must lie on the line $L(t,b_3)$. From this we obtain the expression for $b_3$, namely:
	\begin{align*}
	b_3 &:= \langle \ww(L(t-\frac{2\pi}{3},b_1),L(t+\frac{2\pi}{3},b_2)),\vv(t)\rangle\\
		&=	b_1\langle \vv(t-\frac{2\pi}{3}),\vv(t)\rangle + \frac{b_2 - b_1\langle \vv(t-\frac{2\pi}{3}),\vv(t+\frac{2\pi}{3})\rangle }{\langle \vv(t-\frac{\pi}{6}),\vv(t+\frac{2\pi}{3})\rangle}\langle\vv(t-\frac{\pi}{6}),\vv(t)\rangle\\
	&=	b_1\cos(\frac{2\pi}{3})
	+ \frac{b_2 - b_1\cos(\frac{4\pi}{3}) }{\cos(\frac{5\pi}{6})}\cos(\frac{\pi}{6})\\
	&=	b_1\cos(\frac{2\pi}{3})
	- \left(b_2 - b_1\cos(\frac{4\pi}{3}) \right)\\
	&= - b_1- b_2.
	\end{align*}
	This shows that the lines $L(t-\frac{2\pi}{3}, b_1), L(t+\frac{2\pi}{3},b_2)$ and $L(t, b_3)$ intersect precisely when $b_1+b_2+b_3 =0$.
	
	Define $c:\TT\to \RR$ as $c(t) = a(t-\frac{2\pi}{3})+a(t)+a(t+\frac{2\pi}{3})$, which is a continuous function. Now, we note that, similar to $a$, we have $c(t)\leq -c(t+\pi)$ for all $t$, so that
	$\int_{\TT}c(t)dt\leq -\int_{\TT} c(t+\pi)dt = -\int_{\TT}c(t)dt$, and hence that $\int_{\TT}c(t)dt\leq 0$. We can thus find a $t_1$ such that $c(t_1)\leq 0$. If also $0\leq -c(t_1+\pi)$ then we set $t_0:=t_1$. If instead $-c(t_1+\pi)<0$, we set $t_2 := t_1+\pi$ and obtain $-c(t_2+\pi)=-c(t_1)\geq c(t_1+\pi)> 0$.  By the intermediate value theorem, we then find a $t_0\in \TT$ such that $-c(t_0+\pi) = 0$. Then $c(t_0)\leq -c(t_0+\pi) = 0$.
	
	In both cases, we found $t_0\in \TT$ with $c(t_0)\leq 0 \leq -c(t_0+\pi)$. Now, as moreover $a(t)\leq -a(t+\pi)$ for all $t\in \TT$, we can determine
	\begin{align*}
		b_1 &\in [a(t_0-\frac{2\pi}{3}), - a(t_0+\frac{\pi}{3})],\\
		b_2 &\in [a(t_0+\frac{2\pi}{3}), - a(t_0-\frac{\pi}{3})],\\
		b_3 &\in [a(t_0), - a(t_0+\pi)]
	\end{align*} such that $b_1+b_2+b_3 = 0$.  Indeed, this is possible as the sum of the left-endpoints of the intervals equals $c(t_0)$, whereas the sum of the right-endpoints of the intervals equals $-c(t_0+\pi)$. We now set $\vv_1 := \vv(t_0-\frac{2\pi}{3})$, $\vv_2 := \vv(t_0+\frac{2\pi}{3})$ and $\vv_3 := \vv(t_0)$ and let $\zz_0$ be the unique intersection point of the lines $L(t_0-\frac{2\pi}{3},b_1)$, $L(t_0+\frac{2\pi}{3},b_2)$ and $L(t_0,b_3)$ . As $\zz_0$ lies on each of the three lines, we obtain that $a_i :=\langle \zz_0,\vv_i\rangle = b_i$ for $i=1,2,3$. By the choice of the $b_i$'s in the intervals, it (see \eqref{eq:lmaineq}) now follows that the properties of the lemma are fulfilled. The last line of the lemma follows from the fact that $m_i^L+ m_i^R =\mu(\Omega) + \mu(\{\omega\in \Omega: \langle g(\omega),\vv_i\rangle = a_i\})$.

The result for unbounded $g$ follows by the following reduction to the case of bounded functions.
For $j\in\NN$ let $\Omega_j\subseteq \Omega$ 
be a measurable subset for which $g\chi_{\Omega_j}$ is bounded and with $\Omega_j\uparrow \Omega$. Denote  $\mu_j := \frac{1}{\mu(\Omega_j)}\mu$ and $g_j:=g|_{\Omega_j}\in L_\infty(\Omega_j,\mu_j)$. Applying the result of the lemma to $g_j$, we find $\zz_{0,j}$ and $\vv_{i,j}$ and $a_{i,j}=\langle \zz_{0,j},\vv_{i,j}\rangle$ with the stated properties. The sequence $\zz_{0,j}$ must be bounded. Indeed, otherwise there is an $i\in\{1,2,3\}$ such that for a subsequence of $(a_{i,j})_{j\geq 1}$ we have $a_{i,j}\to+\infty$. However, this would contradict $\frac{1}{2}\leq \mu_j\bigg(\{\omega\in \Omega_j:\ \langle g_j(\omega), \vv_{i,j}\rangle \geq a_{i,j}\}\bigg)$. Thus, by boundedness of the sequences $(\zz_{0,j})_{j\geq 1}$ and $(\vv_{i,j})_{j\geq 1}$, we have for some strictly increasing sequence $(j_k)_{k\geq 1}$ in $\NN$, that the limits $\zz_0:=\lim\limits_{k\to\infty}\zz_{0,j_k}$ and  $\vv_{i}:=\lim\limits_{k\to\infty}\vv_{i,j_k}$ exist. Setting $a_i := \langle \zz_0,\vv_i\rangle$ we also have $a_{i}=\lim\limits_{k\to\infty}a_{i,j_k}$. Using (reversed) Fatou's lemma, we now obtain for $i=1,2,3$ that
	\begin{align*}
		\mu\bigg(\{\omega\in \Omega:\ \langle g(\omega), \vv_i\rangle \leq a_i\}\bigg)
		&\geq\mu\bigg(\bigcap_{K=1}^\infty\bigcup_{k\geq K}\{\omega\in \Omega_{j_k}:\ \langle g(\omega), \vv_{i,j_k}\rangle \leq a_{i,j_k}\}\bigg)\\
		&\geq \limsup\limits_{k\to\infty}\mu\bigg(\{\omega\in \Omega_{j_k}:\ \langle g(\omega), \vv_{i,j_k}\rangle \leq a_{i,j_k}\}\bigg)\\
		&\geq \limsup\limits_{k\to\infty}\mu_{j_k}\bigg(\{\omega\in \Omega_{j_k}:\ \langle g_{j_k}(\omega), \vv_{i,j_k}\rangle \leq a_{i,j_k}\}\bigg)\\
		&\geq \frac{1}{2}.
	\end{align*}
	In the same way $\mu\bigg(\{\omega\in \Omega| \langle g(\omega), \vv_i\rangle \geq a_i\}\bigg)\geq \frac{1}{2}$ can be shown. The last line of the lemma follows as before. This proves the lemma.

\end{proof}

We are now fully equipped to prove \cref{theorem-transformation-with-bound}.
	\begin{proof}[Proof of \cref{theorem-transformation-with-bound}]
	By identifying $\CC$ with $\RR^2$, we can apply Lemma \ref{l2}, to obtain $\zz_0$ and $\vv_1, \vv_2, \vv_3$ and  $a_1, a_2, a_3$ which we will use to prove the result. Without loss of generality we can moreover assume that $\vv_1$,$\vv_2$ and $\vv_3$ are orientated counter-clockwise. In the proof, we distinguish  cases, depending on $n$.
	We prove the result separately for the cases: (1) for $n$ even, or $n=\infty$ and (2) for $n$ odd,
	
	(1). \textit{$n$ is even, or $n=\infty$.}
	
	First, suppose that $n\in\NN$ is even. Then, by the choice of the point $\zz_0$ and of $\vv_1,\vv_2,\vv_3$ (see \cref{l2}) and the fact that $n$ is even, we can for $j=1,2,3$ create partitions $\{I_j^+,I_j^-\}$ of $\Omega_n$ such that $\mu_n(I_j^+) = \frac{\mu_n(\Omega_n)}{2} = \mu_n(I_j^-)$ and such that $\langle g(\omega), \vv_{j}\rangle \leq a_{j}$ whenever $\omega\in I^-_{j}$ and $\langle g(\omega),\vv_{j}\rangle \geq a_{j}$ whenever $\omega\in I_{j}^+$. If instead $n=\infty$ then the same is true, because of the fact that $\mu_n$ is atomless in that case. We can now define the sets
	\begin{align*}
		P_1^+ =   I_1^+ \cap I_2^- \cap I_3^-,&\quad 	P_1^- = I_1^-\cap I_2^+\cap I_3^+,\\
		P_2^+ = I_1^-\cap I_2^+\cap I_3^-, & \quad	P_2^- = I_1^+\cap I_2^- \cap I_3^+,  \\
		P_3^+ = I_1^-\cap I_2^-\cap I_3^+,& \quad  P_3^+ = I_1^+\cap I_2^+ \cap I_3^-, \\
		P_4^+ = I_1^+\cap I_2^+\cap I_3^+, & \quad P_4^- = I_1^- \cap I_2^- \cap I_3^-
	\end{align*}
	that partition $\Omega_n$.
	
	We show that $g(P_4^+\cup P_4^-) \subseteq \{\zz_0\}$. We have that $\vv_1+\vv_2+\vv_3=0$ and therefore $a_1+a_2+a_3=0$.
	For $\omega\in I_1^+\cap I_2^+\cap I_3^+$ we have $\langle g(\omega),\vv_i\rangle\geq a_i, i=1,2,3,$ and $\sum_{i=1}^3\langle g(\omega),\vv_i\rangle =0$. Hence, $\langle g(\omega),\vv_i\rangle=a_i, i=1,2,3$.
	But this means precisely that $g(\omega) = \zz_0$. Similarly $g(P_4^-) \subseteq \{\zz_0\}$. For benefit of the reader, we have visualized the partition sets in  \cref{figure:partition-sets-in-plane}.
	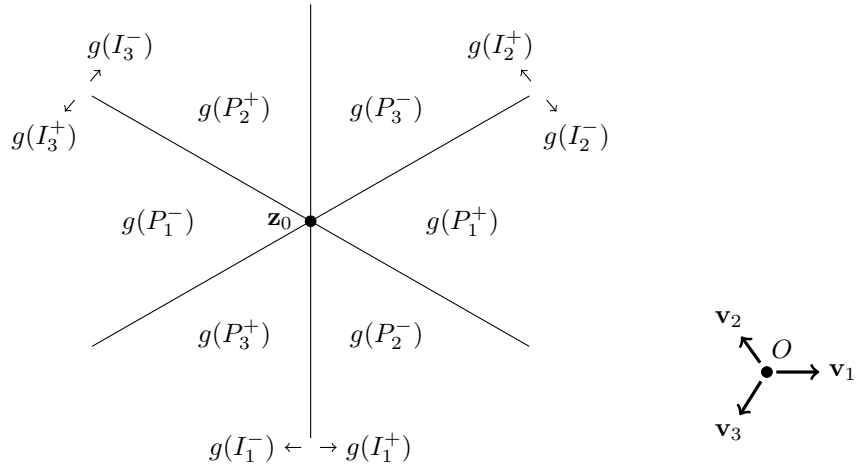
\begin{figure}[h!]
		
		\begin{tikzpicture}[baseline]
			\node at (-0.4,0) 	  	(x0) {$\zz_0$};
			\node at (0,3)   		(L1A) {};
			\node at (0,-3)   	(L1B) {};
			\node at (-3,-2*0.866)   		(L2A) {};
			\node at (3,2*0.866)  	 	(L2B) {};
			\node at (3,-2*0.866)  	 	(L3A) {};
			\node at (-3,2*0.866)   		(L3B) {};
			
			\node at (6,-2)   		(origin) {};
			\node at (6.2,-1.7)   		(O) {$O$};
			\node at (7,-2)   			(v1) {$\vv_1$};
			\node at (5.5,-1.3)   		(v2) {$\vv_2$};
			\node at (5.5,-2.8)   		(v3) {$\vv_3$};
			
			\node at (0.9,-3)   				(I1A) {$g(I_1^+)$};
			\node at (-0.9,-3)   				(I1B) {$g(I_1^-)$};
			\node at (3-0.5,2*0.866+0.6)   		(I2A) {$g(I_2^+)$};
			\node at (3+0.5,2*0.866-0.6)					(I2B) {$g(I_2^-)$};
			\node at (-3 - 0.5 ,2*0.866 - 0.6)				(I3A) {$g(I_3^+)$};
			\node at (-3 + 0.5 ,2*0.866 + 0.6)   		(I3B) {$g(I_3^-)$};
			
			\node at (-2,0)  	 				(P3A) {$g(P_1^-)$};
			\node at (2,0)   					(P3B) {$g(P_1^+)$};
			\node at (1,-1.5)   				(P1A) {$g(P_2^-)$};
			\node at (-1,1.5)   				(P1B) {$g(P_2^+)$};
			\node at (1,1.5)   					(P2A) {$g(P_3^-)$};
			\node at (-1,-1.5)  	 			(P2B) {$g(P_3^+)$};

			\draw[-] (L1A) -- node[xshift =-0.2cm, yshift=-0.2cm] {} (L1B);
			\draw[-] (L2A) -- node[xshift =0.2cm, yshift=-0.2cm] {} (L2B);
			\draw[-] (L3A) -- node[right] {} (L3B);
			
			\draw[->] (L1B) -- node[right] {} (I1A);
			\draw[->] (L1B) -- node[right] {} (I1B);
			
			\draw[->] (L2B) -- node[right] {} (I2A);
			\draw[->] (L2B) -- node[right] {} (I2B);
			
			\draw[->] (L3B) -- node[right] {} (I3A);
			\draw[->] (L3B) -- node[right] {} (I3B);
			
			\draw[->,very thick] (origin) -- node[right] {} (v1);
			\draw[->,very thick] (origin) -- node[right] {} (v2);
			\draw[->,very thick] (origin) -- node[right] {} (v3);
			\filldraw (0,0) circle (2pt);
			\filldraw (6,-2) circle (2pt);
		\end{tikzpicture}
		\caption{The partition sets are visualized for a universal example (any example is like this, except for shifting $\zz_0$ and rotating the lines). The $3$ lines intersect in a single point $\zz_0$. For every line, the set $\Omega_n$ is partitioned in two sets $I_i^{+}$ and $I_i^-$, so that $g(I_i^+)$ and $g(I_i^-)$ lie only on one side of this line. The partition sets $P_j^\pm$ are then such that $g(P_j^\pm)$ lies in one connected component (or its boundary). The sets $g(P_4^+)$ and $g(P_4^-)$ are not visualized. For these we must have $g(P_4^+ \cup P_4^-) \subseteq \{\zz_0\}$.
		}\label{figure:partition-sets-in-plane}
	\end{figure}
	
	We have
	\begin{align}\label{s1}
		\mu_n(P_1^+ \cup P_2^- \cup P_3^- \cup P_4^+)=\mu_n(I_1^+)=\mu_n(I_1^-)=\mu_n(P_1^-\cup P_2^+\cup P_3^+\cup P_4^-),
	\end{align}
	\begin{align}\label{s2}
		\mu_n(P_1^-\cup P_2^+\cup P_3^-\cup P_4^+)=\mu_n(I_2^+)=\mu_n(I_2^-)=\mu_n(P_1^+\cup P_2^-\cup P_3^+\cup P_4^-),
	\end{align}
	\begin{align}\label{s3}
		\mu_n(P_1^-\cup P_2^-\cup P_3^+\cup P_4^+)=\mu_n(I_3^+)=\mu_n(I_3^-)=\mu_n(P_1^+\cup P_2^+\cup P_3^-\cup P_4^-),
	\end{align}
	(\ref{s1})$+$(\ref{s2}):
	$$\mu_n(P_3^-)+\mu_n(P_4^+)=\mu_n(P_3^+)+\mu_n(P_4^-),$$
	(\ref{s1})$+$(\ref{s3}):
	$$\mu_n(P_2^-)+\mu_n(P_4^+)=\mu_n(P_2^+)+\mu_n(P_4^-),$$
	(\ref{s2})$+$(\ref{s3}):
	$$\mu_n(P_1^-)+\mu_n(P_4^+)=\mu_n(P_1^+)+\mu_n(P_4^-).$$
	We thus obtain that $t:=\mu_n(P_j^+) - \mu_n(P_j^-)$ is independent of $j=1,2,3,4$.

	Let us assume that $t\geq 0$ so that $\mu_n(P_j^+)\geq t$. Choose $A_j\subseteq P_{j}^+$ with $\mu_n(A_j) = t$.
	We denote $X_1=(P_4^+\cup P_4^-)\setminus A_4$ and $$X_2^{1,1}=P_1^+\setminus A_1,\quad X_2^{1,2}=P_1^-,\ X_2^{2,1}=P_2^+\setminus A_2,\ X_2^{2,2}=P_2^-,\ X_2^{3,1}=P_3^+\setminus A_3,\ X_2^{3,2}=P_3^-.$$
	
	First, suppose that $n\in \NN$ is even. Then $A_j=\{a_{j,1},\dots,a_{j,l}\},\ j=1,2,3,4,\ l=tn$. Fix $k=1,\dots,l$.
	In each triple $(a_{1,k},a_{2,k},a_{3,k})$ there will be such $i,j\in\{1,2,3\}$ that $\frac{2\pi}{3}\leq\angle g(a_{i,k}),\zz_0, g(a_{j,k})\leq \pi$ (see \cref{figure:partition-sets-in-plane}). Let $\{q\}=\{1,2,3\}\setminus\{i,j\}$.
	Then $\{g(a_{i,k})\}$ and $\{g(a_{j,k})\}$, and also $\{g(a_{q,k})\}$ and $\{g(a_{4,k})\}$ form pairs of $(\zz_0,\frac{\pi}{3})$-conjugate sets.
	We put $X_2^{2k+2,1}=\{a_{i,k}\},\ X_2^{2k+2,2}=\{a_{j,k}\},\ X_2^{2k+3,1}=\{a_{q,k}\},\ X_2^{2k+3,2}=\{a_{4,k}\}$ and $X_2^{m,1}=X_2^{m,2}=\emptyset$ for $m\geq 2l+4$.
	
	We assume now that $n=\infty$. Let $\Sigma_j=\{Y_j^1,Y_j^2,\dots\}$ be a maximal system of pairwise disjoint measurable subsets of $A_j,\ j=1,2,3,4$, such that $\mu_\infty(Y_1^k)=\mu_\infty(Y_2^k)=\mu_\infty(Y_3^k)=\mu_\infty(Y_4^k)>0$ and the four $(g(Y_1^k),g(Y_2^k),g(Y_3^k),g(Y_4^k))$ is divided into two pairs of $(\zz_0,\frac{\pi}{3})$-conjugate sets for $k=1,2,\dots$.
	
	Put $B_j=A_j\setminus\bigcup_kY_j^k$. Then $\mu_\infty(B_1)=\mu_\infty(B_2)=\mu_\infty(B_3)=\mu_\infty(B_4)=t_0$. Suppose that $t_0>0$. If the sets $g(B_1),g(B_2),g(B_3)$ are located on three rays emanating from $\zz_0$ and forming angles $\frac{2\pi}{3}$ then $g(B_1),g(B_2)$ are $(\zz_0,\frac{\pi}{3})$-conjugate sets and the same for $g(B_3),g(B_4)$. This contradicts the maximality of the above set systems $\Sigma_j$.
	
	Otherwise, there will be such $b_1\in B_i,\ b_2\in B_j,\ i,j\in\{1,2,3\},\ i\neq j,$ that $\angle g(b_1),\zz_0, g(b_2)>\frac{2\pi}{3}$ and $g(b_1),g(b_2)$ are essential values of $g|_{B_i\cup B_j}$.
	Then there will be such neighborhoods $V_1$ and $V_2$ of the points $g(b_1)$ and $g(b_2)$, respectively, that $V_1,V_2$ are $(\zz_0,\frac{\pi}{3})$-conjugate sets. Therefore there exist sets $Y_1\subset B_i,\ Y_2\subset B_j$ so that $\mu_\infty(Y_1)=\mu_\infty(Y_2)>0$ and $g(Y_k)\subset V_k,\ k=1,2$. Hence, $g(Y_1),g(Y_2)$ are $(\zz_0,\frac{\pi}{3})$-conjugate sets. Let $\{q\}=\{1,2,3\}\setminus\{i,j\}$.
	There exists $Y_3\subset B_q,\ Y_4\subset B_4,\ \mu_\infty(Y_3)=\mu_\infty(Y_4)=\mu_\infty(Y_1)$. It is clear that $g(Y_3),g(Y_4)$ are $(\zz_0,\frac{\pi}{3})$-conjugate sets.
	The presence of sets $Y_1,Y_2,Y_3,Y_4$ contradicts the maximality of the above systems $\Sigma_j$.
	
	The contradiction obtained in both cases shows $t_0=0$. Therefore the system\\ $\{X_1\}\cup \{X_2^{m,i}: 1\leq m\leq 3, 1\leq i\leq 2\}$ can be completed using $\Sigma_j,\ j=1,2,3,4$.
	
	It remains to define $T$ so that $T_{X_{1}}=Id_{X_{1}},\ T(X_2^{m,1})=X_2^{m,2},\ T(X_2^{m,2})=X_2^{m,1}$ for  $m=1,2,\dots$ and such that $T^2=\Id_{\Omega_n}$. Then the inequality (\ref{e1}) follows from the Lemma \ref{l1}.
	
	The case that $t\leq 0$ is similar, by changing the roles of $P_j^+$ and $P_j^-$.

	(2). \textit{$n$ is odd.}
	
	We can for $i=1,2,3$ instead build partitions $\{I_i^+, \{\omega_i\}, I_j^-\}$ of $\Omega_n$ with $\mu_n(I_i^+) = \mu_n(I_i^-)$ and such that  $\langle g(\omega), \vv_i\rangle \leq a_i$ whenever $\omega\in I_i^-$ and $\langle g(\omega),\vv_i\rangle \geq a_i$ whenever $\omega\in I^+$ and $\langle g(\omega_i),\vv_i\rangle = a_i$. Indeed, such $\omega_i$ exist because $|\{\omega\in \Omega_n:\ \langle g(\omega), \vv_i\rangle \leq a_i\}|,|\{\omega\in \Omega_n:\ \langle g(\omega), \vv_i\rangle \geq a_i\}|\geq \frac{n+1}{2}$ and therefore $\{\omega\in \Omega_n:\ \langle g(\omega), \vv_i\rangle \leq a_i\}\cap \{\omega\in \Omega_n:\ \langle g(\omega), \vv_i\rangle \geq a_i\}\neq \emptyset$. Denote $Y_0 = \{\omega_1,\omega_2,\omega_3\}$.
	
	Now, suppose that $\zz_0\in g(\Omega_n)$. Then we could have chosen $\omega_1 = \omega_2=\omega_3$ all equal and such that $g(\omega_i) = \zz_0$. Then $|Y_0| = 1$ and the sets $\{I_i^+,I_i^-\}$ are all partitions of $\Omega_n\setminus Y_0$ similar to (1), and we can build the corresponding automorphism $T$ of $\Omega_n\setminus Y_0$. This completes the proof for that case by setting $T(\omega_{1}) = \omega_{1}$.
	
	We can thus assume that $\zz_0\not\in g(\Omega_n)$ so that in particular $g(\omega_i)\not= g(\omega_j)$ for $i\not=j$ and $|Y_0| = 3$. Now suppose first that $\zz_0\in \Conv(g(Y_0))$. For all $i\in \{1,2,3\}$ we then have that $Y_0\cap I_i^+$ and $Y_0\cap I_i^-$ both consist of $1$ element.
	Hence, $\mu_n(I_i^+\setminus Y_0) = \mu_n(I_i^-\setminus Y_0)$ and the partitions $\{I_i^+\setminus Y_0 ,I_i^-\setminus Y_0\}$ of $\Omega_n\setminus Y_0$ satisfy the same properties as (1). We thus obtain a measure preserving automorphism $T$ of $\Omega_n\setminus Y_0$ with the same properties. Now we can set $T(\omega_1)=\omega_2$, $T(\omega_2) = \omega_3$ and $T(\omega_3)=\omega_1$, so that $\angle g(\omega_i),\zz_0,g(T(\omega_i)) = \frac{2\pi}{3}$. This finishes the proof by \cref{l1}
	
	Now suppose that $\zz_0\not\in \Conv(g(Y_0))$.
	Then it can be seen geometrically (for intuition see \cref{figure:five-cycle}), that there is a unique choice of (distinct) indices $i_1,i_2,i_3\in \{1,2,3\}$ such that \begin{align}\label{eq:unique-choice-of-indices}
		\{\omega_{i_1}\} = Y_0\cap I_{i_2}^- &\quad  \{\omega_{i_3}\} = Y_0\cap I_{i_2}^+.
	\end{align} Now, suppose that $\omega_{i_1}\not\in I_{i_3}^+$. Then as $\omega_{i_1}\not=\omega_{i_3}$ we get $\omega_{i_1}\in I_{i_3}^-$. But then as $\omega_{i_1}\in \{\omega_{i_1}\}\cap I_{i_2}^-\cap I_{i_3}^-$ we would get $g(\omega_{i_1}) = \zz_0$ by the same argument as why $g(P_4^+\cup P_4^-)\subseteq \{\zz_0\}$ in (1). However, $\zz_0\not\in g(\Omega_n)$ by our assumption so this cannot be the case. We conclude that we must have $\omega_{i_1}\in I_{i_3}^+$. By a same argument we find that we must have $\omega_{i_3}\in I_{i_1}^-$ (Indeed, otherwise $\omega_{i_3}\in I_{i_1}^+$ so that $\omega_{i_3}\in I_{i_1}^+ \cap I_{i_2}^+\cap \{\omega_{i_3}\}$, which would imply $g(\omega_{i_3})=\zz_0$, which gives a contradiction). Furthermore, we claim that $\omega_{i_2}\in I_{i_3}^+$. Indeed, if $\omega_{i_2}\in I_{i_3}^-$ then we could rearrange the indexes as $i_1' = i_2$, $i_2' = i_3$ and $i_3' = i_1$, so that we get
	$\{w_{i_1'}\}  =\{w_{i_2}\} = Y_0\cap I_{i_3}^- = Y_0\cap I_{i_2'}^-$
	and 	$\{w_{i_3'}\}  =\{w_{i_1}\} = Y_0\cap I_{i_3}^+ = Y_0\cap I_{i_2'}^+$.
	This contradicts the uniqueness of the choice $i_1,i_2,i_3$ satisfying \eqref{eq:unique-choice-of-indices}. We conclude that indeed $\omega_{i_2}\in I_{i_3}^+$.
	By the same argument we find $\omega_{i_2}\in I_{i_1}^-$ (Indeed, if $\omega_{i_2}\in I_{i_1}^+$ we could take the rearrangement $i_1' = i_3$, $i_2' = i_1$ and $i_3' = i_2$ to obtain $\{w_{i_1'}\}  =\{w_{i_3}\} = Y_0\cap I_{i_1}^- = Y_0\cap I_{i_2'}^-$ and $\{w_{i_3'}\}  =\{w_{i_2}\} = Y_0\cap I_{i_1}^+ = Y_0\cap I_{i_2'}^+$, which contradicts the uniqueness).  For clarity we summarize the results:
	\begin{align*}
		\{\omega_{i_1}\} = Y_0\cap I_{i_2}^- &\quad \{\omega_{i_3}\} = Y_0\cap I_{i_2}^+,\\
		\{\omega_{i_2},\omega_{i_3}\} = Y_0 \cap I_{i_1}^-&\quad  \{\omega_{i_1},\omega_{i_2}\} = Y_0\cap I_{i_3}^+.
	\end{align*}
	We now obtain
	\begin{align}
		\label{eq: sum-i1} \mu_n(I_{i_1}^+\cap I_{i_2}^-)  + \mu_n(I_{i_1}^+\cap I_{i_2}^+) &= \mu_n(I_{i_1}^+\setminus \{\omega_{i_2}\})= \mu_n(I_{i_1}^+),\\
		\label{eq: sum-i2a}\mu_n(I_{i_1}^+\cap I_{i_2}^-)  +\mu_n(I_{i_1}^-\cap I_{i_2}^-) &= \mu_n(I_{i_2}^-\setminus \{\omega_{i_1}\}) = \mu_n(I_{i_2}^-)-\frac{1}{n},\\
		\label{eq: sum-i3}\mu_n(I_{i_2}^+\cap I_{i_3}^-)  +\mu_n(I_{i_2}^-\cap I_{i_3}^-) &= \mu_n(I_{i_3}^-\setminus \{\omega_{i_2}\}) = \mu_n(I_3^{-}),\\
		\label{eq: sum-i2b}\mu_n(I_{i_2}^+\cap I_{i_3}^-)  +\mu_n(I_{i_2}^+\cap I_{i_3}^+) &= \mu_n(I_{i_2}^+\setminus \{\omega_{i_3}\}) = \mu_n(I_{i_2}^+)-\frac{1}{n}.
	\end{align}
	Hence, by \eqref{eq: sum-i1} + \eqref{eq: sum-i2a} we obtain $\mu_n(I_{i_1}^+\cap I_{i_2}^+) = \frac{1}{n} + \mu_n(I_{i_1}^-\cap I_{i_2}^-)$
	and by summing up \eqref{eq: sum-i3}  with \eqref{eq: sum-i2b} we obtain $\mu_n(I_{i_2}^-\cap I_{i_3}^-) = \frac{1}{n}+ \mu_n(I_{i_3}^+\cap I_{i_2}^+)$. We conclude the existences of  $\omega_4 \in I_{i_1}^+\cap I_{i_2}^+$ and $\omega_5 \in I_{i_2}^-\cap I_{i_3}^-$. Now, for the sets $P_4^+ := I_{i_1}^+\cap I_{i_2}^+\cap I_{i_3}^+$ and $P_4^- := I_{i_1}^-\cap I_{i_2}^-\cap I_{i_3}^-$ we have  that $g(P_4^+\cup P_4^-) \subseteq \{\zz_0\}$ (same as in (1)), and hence $P_4^+\cup P_4^-=\emptyset$ as $\zz_0\not\in g(\Omega_n)$ by assumption. This means that $\omega_4\not\in I_{i_3}^+$ and $\omega_5\not\in I_{i_1}^-$. Also, as $\omega_{i_3}\in I_{i_1}^-$ and $\omega_{i_1}\in I_{i_3}^+$ we get that $\omega_4 \not=\omega_{i_3}$ and $\omega_5\not= \omega_{i_1}$. As $\{I_i^+,\{\omega_i\},I_i^-\}$ are partitions of $\Omega_n$, we conclude that $\omega_4\in I_{i_1}^+\cap I_{i_2}^+\cap I_{i_3}^-$ and $\omega_5\in I_{i_1}^+\cap I_{i_2}^-\cap I_{i_3}^-$  Denote $Y_1 = \{\omega_1,\omega_2,\omega_3,\omega_4,\omega_5\}$, so that by the above we have $|Y_1| = 5$ and moreover:
	\begin{align*}
		Y_1 \cap I_{i_1}^- &= \{\omega_{i_2},\omega_{i_3}\}, &\quad Y_1 \cap I_{i_2}^- &= \{\omega_{i_1},\omega_5\},  &\quad  Y_1 \cap I_{i_3}^- &=  \{\omega_{4},\omega_{5}\},\\
		Y_1 \cap I_{i_1}^+ &= \{\omega_4,\omega_5\}, &\quad  Y_1 \cap I_{i_2}^+ &= \{\omega_{i_3},\omega_4\}, &\quad Y_1\cap I_{i_3}^+ &=  \{\omega_{i_1},\omega_{i_2}\}.
	\end{align*}
	
	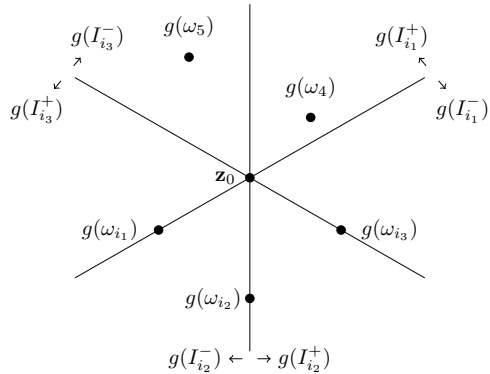
\begin{figure}[h!]
		\scalebox{0.8}{
			\begin{tikzpicture}
				\node at (-0.4,0) 	  	(z0) {$\zz_0$};
				\node at (0,3)   		(L1A) {};
				\node at (0,-3)   	(L1B) {};
				\node at (-3,-2*0.866)   		(L2A) {};
				\node at (3,2*0.866)  	 	(L2B) {};
				\node at (3,-2*0.866)  	 	(L3A) {};
				\node at (-3,2*0.866)   		(L3B) {};
				
				\node at (-2.3,-1*0.866)   		(w-i1) {$g(\omega_{i_1})$};
				\node at (-0.6,-2)   	(w-i2) {$g(\omega_{i_2})$};
				\node at (2.3,-1*0.866)   		(w-i1) {$g(\omega_{i_3})$};
				\node at (1,1.5)   		(w-i1) {$g(\omega_{4})$};
				\node at (-1,2.5)   		(w-i1) {$g(\omega_{5})$};
				
				\node at (0.9,-3)   				(I1A) {$g(I_{i_2}^+)$};
				\node at (-0.9,-3)   				(I1B) {$g(I_{i_2}^-)$};
				\node at (3-0.5,2*0.866+0.6)   		(I2A) {$g(I_{i_1}^+)$};
				\node at (3+0.5,2*0.866-0.6)					(I2B) {$g(I_{i_1}^-)$};
				\node at (-3 - 0.5 ,2*0.866 - 0.6)				(I3A) {$g(I_{i_3}^+)$};
				\node at (-3 + 0.5 ,2*0.866 + 0.6)   		(I3B) {$g(I_{i_3}^-)$};

				\draw[-] (L1A) -- node[xshift =-0.2cm, yshift=-0.2cm] {} (L1B);
				\draw[-] (L2A) -- node[xshift =0.2cm, yshift=-0.2cm] {} (L2B);
				\draw[-] (L3A) -- node[right] {} (L3B);
				
				\draw[->] (L1B) -- node[right] {} (I1A);
				\draw[->] (L1B) -- node[right] {} (I1B);
				
				\draw[->] (L2B) -- node[right] {} (I2A);
				\draw[->] (L2B) -- node[right] {} (I2B);
				
				\draw[->] (L3B) -- node[right] {} (I3A);
				\draw[->] (L3B) -- node[right] {} (I3B);
				\filldraw (0,0) circle (2pt);
				\filldraw (0,-2) circle (2pt);
				\filldraw (-1.5,-0.866) circle (2pt);
				\filldraw (1.5,-0.866) circle (2pt);
				\filldraw (-1,2) circle (2pt);
				\filldraw (1,1) circle (2pt);
		\end{tikzpicture}}
		\caption{The $5$ points are depicted for an example.}
		\label{figure:five-cycle}
	\end{figure}
	Now, as all these sets have size $2$, we must have that $$\mu_n(I_i^+\setminus Y_1) =\mu_n(I_i^+)-\frac{2}{n} = \mu_n(I_i^-)-\frac{2}{n} = \mu_n(I_i^-\setminus Y_1).$$ This means that the partitions $\{I_i^+\setminus Y_1,I_i^-\setminus Y_1\}$ of $\Omega_n\setminus Y_1$ satisfy the same properties as in (1). We can therefore find a transformation $T$ of $\Omega_n\setminus Y_1$ with the same properties.  We can now define $T$ on $Y_1$ by setting  $T(\omega_{i_1}) = \omega_4$, $T(\omega_4) = \omega_{i_2}$, $T(\omega_{i_2}) = \omega_5$, $T(\omega_5) = \omega_{i_3}$ and $T(\omega_{i_3}) = \omega_{i_1}$. Then $\angle g(\omega_i),\zz_0,g(T(\omega_i))\geq \frac{2\pi}{3}$ for all $i$.
	Appealing to \cref{l1} this impies $|g(w)-g(T(w))| \leq \frac{\sqrt{3}}{2}(|g(w)-\zz_0| + |g(T(w)-\zz_0|)$ for all $\omega\in \Omega$, which shows that \eqref{e1} holds true.
	The inequality \eqref{e2} follows from it.
Furthermore,  in each of the considered cases it clear how to split $\Omega_n$ into the parts $X_1,X_2,X_3,X_5$ (note that by construction  we have $T^k(\omega)=\omega$ for some $k\in \{1,2,3,5\}$ for $\omega\in \Omega$). We prove the final statement.

Let $n<\infty$ and let $(T_m)$, $(z_m)$ be sequences such that $0<\Lambda(g,T_m,z_m)\uparrow \Lambda(g)$. Then
$$\Lambda(g,T_m,z_m)^{-1}|g(\omega)-g(T_m(\omega))|\geq|g(\omega)-z_m|+|g(T_m(\omega))-z_m|\geq|g(\omega)-z_m|$$ for any $\omega\in\Omega_n$.
Let $\omega_0\in\Omega_n$. Then
$$|g(\omega_0)-z_m|\leq \Lambda(g,T_m,z_m)^{-1}\Diam(g(\Omega_n))\leq \Lambda(g,T_1,z_1)^{-1}\Diam(g(\Omega_n)).$$
Since $card(Aut_n)=n!<\infty$ and since $\{z\in\CC:\ |g(\omega_0)-z|\leq \Lambda(g,T_1,z_1)^{-1}\Diam(g(\Omega_n))\}$ is compact, there exists a sequence $(m_k)$ and a $z_0\in\CC$ such that 
$$z_0=\lim_k z_{m_k},\ T_0:=T_{m_1}=T_{m_2}=\dots=T_{m_k}=\dots\ .$$
Then $\Lambda(g,T_0,z_0)=\Lambda(g)$.
\end{proof}

\section{Commutator estimates for normal operators in finite factors}\label{section:estimates-in-finite-factors}
The main result of this section, \cref{t3} below,  establishes the commutator inequality \eqref{bs_t1_2_2} for normal element $a\in S(\calM)$, where $\calM$ is a finite factor, and provides upper and lower bounds on the optimal constant $C_{\calM}$. This yields a version of \cite[Theorem 1.1]{BBS}, suitable for normal elements.
We  consider the case of I$_n$-factors ($n<\infty$) in \cref{I_n}) and  the case of II$_1$-factors in \cref{t_211} and show that the commutator inequality holds for the constant $\frac{\sqrt{3}}{2}$. The proof for II$_1$-factors requires two additional results, \cref{t_21} and \cref{l_conj}. Furthermore, in order to prove the upper bounds in \cref{t3} we provide \cref{prop:restrict-to-permutation-matrix}. 

\begin{theorem}\label{I_n}
	Let $\calM=B(\calH)$ be an $I_n$-factor for $n\in\NN$. For an arbitrary normal operator $a\in \calM$ there is a unitary $u\in \calU(\calM)$ and a $z_0\in \CC$  such that
    \begin{align}\label{I_n_1}
      |[a,u]|\geq \frac{\sqrt{3}}{2}(|a - z_0\textbf{1}| + u|a-z_0\textbf{1}|u^*).
    \end{align}
    Moreover, $u$ can be chosen so that 
    \begin{itemize}
    	\item when $n$ is even there are projections $p_1,p_2$ such that $p_1+p_2=\textbf{1}$
    	\item when $n$ is odd there are projections $p_1,p_2,p_3,p_5$ such that $p_1+p_2+p_3+p_5=\textbf{1}$
    \end{itemize} so that
    $$p_ku=up_k,\ u^kp_k=p_k,\ k=1,2,3,5.$$
If $a\in \calM$ is such that its spectrum $\sigma(a)$ lies on a straight line, then we can obtain true equality:
    \begin{align}\label{I_n_2}
    |[a,u]|=|a - z_0\textbf{1}| + u|a-z_0\textbf{1}|u, \text{ for some } u^*=u\in \calU(\calM), z_0\in \CC.
    \end{align}
We remark that when $n=1,2$ every normal $a\in \calM$ satisfies this extra condition.
\end{theorem}
\begin{proof}
	Since $a$ is a normal element on an $n$-dimensional Hilbert space, it follows from the spectral mapping theorem that there is a unitary $U:\calH\to L_2(\Omega_n)$ such that $a = U^*M_{g}U$, where $M_g$ is the multiplication operator on $L_2(\Omega_n)$ for some $g\in L_\infty(\Omega_n)$. Applying \cref{theorem-transformation-with-bound} to $g$, we find a transformation $T$ and a $z_0\in \CC$ such that
	\begin{align} \label{eq:I_n-function-inequality}
		|g\circ T - g| \geq \frac{\sqrt{3}}{2}(|g-z_0| + |g\circ T-z_0|)
	\end{align} together with the given partition of $\Omega_n$ consisting of the sets $X_1,X_2$ (when $n$ is even)  or $X_1,X_2,X_3,X_5$ (when $n$ is odd) and that satisfy $T^k|_{X_k} = \Id_{X_k}$. Now let $u_{T}$ be the Koopman operator on $L_2(\Omega_n)$ corresponding to $T$, i.e. $u_Tf = f\circ T$. Denote $u = U^*u_TU$. Then
	\begin{align*}
		|[a,u]|  &= |u(u^*au - a)| \\
		&= |uau^* - a| \\
		&= U^*|u_TM_gu_T^*-M_g|U \\
		&= U^*|M_{g\circ T} - M_g|U \\
		&= U^*M_{|g\circ T - g|}U\\
		&\geq \Lambda_n\left(U^*|M_g - z_0|U + U^*|M_{g\circ T} - z_0|U\right)\\
		 &= \Lambda_n\left(U^*|M_{g} - z_0|U + U^*u_{T}|M_{g} - z_0|u_{T}^*U\right).\\
		 &= \Lambda_n\left(|a-z_0| + u|a-z_0|u^*\right)
	\end{align*}
	We now define the projections by setting $p_k = U^*\chi_{X_k}U$ which clearly satisfy the statements.
	
    If $\sigma(a)$ lies on a straight line, then there exist scalars $\alpha,\beta\in\CC,\ |\alpha|=1,$ such that $a_1:=\alpha(a-\beta\textbf{1})\in \calM$ is self-adjoint.
     It follows from  \cref{bs_t1} that
    there exist  $z_0\in\RR$ and $u=u^*\in\calU(\calM)$ that 
    $$|[a,u]|=|[a_1,u]|=|a_1 - z_0\textbf{1}| + u|a_1-z_0\textbf{1}|u=|a - (\beta+z_0\alpha^{-1})\textbf{1}| + u|a-(\beta+z_0\alpha^{-1})\textbf{1}|u.$$
\end{proof}
We need the following result, which for a diffuse semi-finite von Neumann algebra $(\calM,\tau)$ and a normal measurable $a\in S(\calM)$ establishes an injective $*$-homomorphism $F$ between $S[0,1]$ and $S(\calM)$ which preserves measure and is such that $a$ lies in the image of $F$.
Special cases of the result which follows for positive bounded elements of $\calM$ and positive elements of $L_1(M,\tau)$ can be found in \cite[Lemma 9]{DSZ} and in \cite[Lemma 4.1]{CS} respectively.

\begin{theorem}\label{t_21}
    Let $\calM$ be a diffuse (i.e. atomless) von Neumann algebra with a faithful normal tracial state $\tau$, let $a\in S(\calM)$ be a normal operator. There exists such an injective $*$-homomorphism $F:S[0,1]\to S(\calM)$ such that $a\in \Image(F)$ and $m(A)=\tau(F(\chi_A))$ for any measurable subset $A\subset [0,1]$ (here $m$ is the Lebesgue measure on $[0,1]$).
\end{theorem}
\begin{proof}
    Let $e$ be a spectral measure of the operator $a$ defined on
    the $\sigma$-algebra $\calB(\sigma(a))$ of Borel subsets in $\sigma(a)$. Then $\tau(e(\cdot))$ yields a probability measure on $\calB(\sigma(a))$.
    By the spectral theorem (see \cite[Theorem 13.33]{Rudin}), we have
    $$a=\int_{\sigma(a)}\lambda de(\lambda).$$

    Let $X_0$ be a set of eigenvalues $a$.  It is clear that $X_0\subset\sigma(a)$ and $card(X_0)\leq\aleph_0$. Indeed, if $t\in X_0$ then $e(\{t\})\neq 0$ and
        $\sum_{t\in X_0}\tau(e(\{t\}))=\tau(e(X_0))\leq 1$.
	Let $t\in X_0$. Since $\calM$ is diffuse, it follows that in $\calM$ there is such a chain of projections $f^t_s\uparrow_s e(\{t\})$ such that $\tau(f^t_s)=s,\ s\in Y_t:=[0,\tau(e(\{t\}))]$.

    Denote by $f_t$ the spectral measure by $\calB(Y_t)$ given by the equality $$f_t((s_1,s_2))=f^t_{s_2}-f^t_{s_1}.$$
    We have $\tau(f_t(A))=m(A)$ for any $a\in\calB(Y_t)$.
    Let us now set
    $$X=(\sigma(a)\setminus X_0)\sqcup \bigsqcup_{t\in X_0}Y_t.$$
    On $\calB(X)$, we define a spectral measure $g$ such that
    $$g|_{\calB(\sigma(a)\setminus X_0)}=e|_{\calB(\sigma(a)\setminus X_0)},\ g|_{\calB(Y_t)}=f_t,\ t\in X_0,$$
    and a scalar measure
    $$\mu_X(A)=\tau(e(A\cap (\sigma(a)\setminus X_0)))+\sum_{t\in X_0}\mu(A\cap Y_t).
    $$

    It follows that $(X,\calB(X),\mu_X)$ is a Lebesgue space with an atomless probability measure.  Hence, it is isomorphic to the segment
    $[0,1]$ equipped with Lebesgue measure $m$,  see e.g. \cite[Theorem 9.5.1]{BOG2}.

    A linear mapping $F:S(X,\calB(X),\mu_X)\to S(\calM)$ is defined  by 
    $$F(\varphi)=\int_X \varphi(x)dg(x)$$ for any $\varphi\in S(X,\calB(X),\mu_X)$ (see \cite[Definition 1.5.6]{DdPS}). We remark that  $F(\chi_A)=g(A)$ for measurable $A\subseteq X$ and that $\mu_{X}(A) = \tau(F(\chi_{A}))$. 
    Furthermore $F(\chi_{A}\chi_{B}) = F(\chi_{A\cap B}) = g(A\cap B) = g(A)g(B) = F(\chi_{A})F({\chi_B})$ for measurable sets $A,B\subseteq X$. Therefore, as $F$ is continuous with respect to the topologies of convergence in measure in  $S(X,\calB(X),\mu_{X})$ and $S(\calM,\tau)$ and since simple functions in $S(X,\calB(X),\mu_{X})$ are dense with respect to the measure topology, it follows that $F(\varphi\psi) = F(\varphi)F(\psi)$ for all $\varphi,\psi \in S(X,\calB(X),\mu_X)$.
    Moreover, $F(\overline{\varphi})=\int_X \overline{\varphi(x)}dg(x)=F(\varphi)^*$ so we find that $F$ is a $*$-homomorphism. Now, suppose $\varphi \in S(X,\calB(X),\mu_{X})$ is such that $F(\varphi)=0$ and $B\subseteq X$ is such that $\varphi(x)\not=0$ for a.e. $x\in B$. Then 
    $g(B) = F(\chi_{B}) = F(\frac{1}{\varphi}\chi_{B})F(\varphi)=0$, thus $\mu_{X}(B)=\tau(g(B))=0$. This shows that $F$ is injective.

    Finally, let us define the function $f$ by setting $f(t)=t$ for $t\in \calB(\sigma(a)\setminus X_0)$ or $t\in Y_t$. Then $f\in S(X,\calB(X),\mu_X)$ and $F(f)=a$.
\end{proof}

\begin{lemma}\label{l_conj}
Let $\calM$ be a finite von Neumann algebra, let $a,b\in S(\calM)$ be normal operators, $z_0\in\mathbb{C},\ 0\leq\alpha<\pi$ and let $\sigma(a),\ \sigma(b)$ be $(z_0,\alpha)$-conjugate sets.
Then 
\begin{align}\label{eq:lemma-conjugate-sets-eq}
	v|a-b|v^*\geq (|a-z_0\textbf{1}|+|b-z_0\textbf{1}|)\cos\frac{\alpha}{2}
\end{align}
for some $v\in \calU(\calM)$.
\end{lemma}
\begin{proof}
Since $\sigma(a)$ and $\sigma(b)$ are ($z_0$,$\alpha$)-conjugate, the shifted sets $\sigma(a)-z_0$ and $\sigma(b)-z_0$ are ($0$,$\alpha$)-conjugate. We can then obtain a pair of lines as in \cref{fig:conjugate-sets}, intersecting at the origin with an angle $\alpha$. 
By rotating the complex plane around the origin we can assure that these lines are symmetric with respect to the real axis. This is to say that there exists a function $f(z)=c(z-z_0)$ with $|c|=1$ so that
$$f(\sigma(a))\subset\{z:\ -\frac{\alpha}{2}\leq Arg(z)\leq \frac{\alpha}{2}\},\ f(\sigma(b))\subset\{z:\ \pi-\frac{\alpha}{2}\leq Arg(z)\leq \pi+\frac{\alpha}{2}\}.$$
Let $a_1=f(a),\ b_1=f(b)$. We have
$$|a_1|\cos\frac{\alpha}{2}\leq\Re a_1,\ |b_1|\cos\frac{\alpha}{2}\leq-\Re b_1.$$
Therefore
\begin{align}\label{eq:lemma-conjugate-sets-eq1}
 (|a-z_0\textbf{1}|+|b-z_0\textbf{1}|)\cos\frac{\alpha}{2}=(|a_1|+|b_1|)\cos\frac{\alpha}{2}\leq \Re a_1-\Re b_1
 =\Re(a_1-b_1)\leq\Re(a_1-b_1)_+
\end{align}
By \cref{p_AAP_21}, we obtain 
\begin{align}\label{eq:lemma-conjugate-sets-eq2}
	\Re(a_1-b_1)_+\leq v|a_1-b_1|v^*=v|a-b|v^*.
\end{align}
for some $v\in \calM$ with $v^*v=\mathbf{1}$. Since $\textbf{1}$ is a finite projection it follows that $vv^*=\textbf{1}$, i.e. $v\in \calU(\calM)$.
Combining \eqref{eq:lemma-conjugate-sets-eq1} and \eqref{eq:lemma-conjugate-sets-eq2} establishes \eqref{eq:lemma-conjugate-sets-eq}
\end{proof}

We now prove a version of  \cref{I_n} for II$_1$-factors. Equation \eqref{t_211_1} below is slightly different from \eqref{I_n_1} as it involves a second unitary $w\in \calU(\calM)$.

\begin{theorem}\label{t_211}
    Let $\calM$ be a factor of type II$_1$, $a\in S(\calM)$ be normal. Then there exists a $z_0\in\mathbb{C}$, $u=u^*\in \calU(\calM)$ and $w\in \calU(\calM)$ so that
    \begin{align}\label{t_211_1}
        w|[a,u]|w^*\geq \frac{\sqrt{3}}{2}\cdot(|a-z_0\textbf{1}|+u|a-z_0\textbf{1}|u).
    \end{align}
    If $\sigma(a)$ lies on a straight line then
    \begin{align}\label{t_211_11}
    |[a,u]|=|a-z_0\textbf{1}|+u|a-z_0\textbf{1}|u.
    \end{align}

\end{theorem}
\begin{proof}
 Let $\tau$ be a faithful normal tracial state on $\calM$ and let $F:S[0,1]\to S(\calM)$ be an injective $*$-homomorphism from Theorem \ref{t_21} satisfying $a\in \Image(F)$ . Let $g=F^{-1}(a)$.

 It follows from \cref{theorem-transformation-with-bound} that there exists $z_0$ such that $[0,1]$ can be divided into disjoint  measurable parts $\{X_1\}\cup \{X_2^{m,i}: m\geq 1, 1\leq i\leq 2\}$ so that $g(X_1)\subset\{z_0\},\ \mu(X_2^{m,1})=\mu(X_2^{m,2})$ and the sets $g(X_2^{m,1}),\ g(X_2^{m,2})$ are $(z_0,\frac{\pi}{3})$-conjugate for $m=1,2,\dots$ (where $\mu$ is the Lebesgue measure on $[0,1]$).

 Let $e=F(\chi_{X_1}),\ p_m=F(\chi_{X_2^{m,1}}),\ q_m=F(\chi_{X_2^{m,2}}),\ m=1,2,\dots$. Then $p_m\sim q_m,\ m=1,2,\dots$, since $\tau(p_m)=\mu(X_2^{m,1})=\mu(X_2^{m,2})=\tau(q_m)$. Besides $e+\sum_{m\geq 1}(p_m+q_m)=\textbf{1}$.
 Hence, there exists such $u=u^*\in \calU(\calM)$ that
 $$ue=e,\ up_m=q_mu,\ m=1,2,\dots.$$
 Note also that $p_mu=uq_m$ since $u$ self-adjoint.
 It is clear that
 $$|[a,u]|e=|[a-z_0\textbf{1},u]|e=0=(|a-z_0\textbf{1}|+u|a-z_0\textbf{1}|u)e.$$
 For any $m=1,2,\dots$ $\sigma(ap_m)$ coincides with the set $A_m$ of essential values of the function $g|_{X_{2}^{m,1}}$ and $\sigma(uaup_m)=\sigma(aq_m)$ coincides with the set $B_m$ of essential values of the function $g|_{X_{2}^{m,2}}$ (here the operators $ap_m$ and $uaup_m$ are considered as elements of the algebra $p_m\calM p_m$).
 The sets $A_m$ and $B_m$ are $(z_0,\frac{\pi}{3})$-conjugate sets. It follows from the Lemma \ref{l_conj} that
 \begin{align}\label{t_211_2}
v_m|a-uau|v_m^*p_m=v_m|a-uau|p_mv_m^*\geq \frac{\sqrt{3}}{2}\cdot(|a-z_0\textbf{1}|+u|a-z_0\textbf{1}|u)p_m
\end{align}
 for some $v_m\in \calU(p_m\calM p_m)$.

 Applying the automorphism $u\cdot u$ to (\ref{t_211_2}), and noting that $u|a-uau|u = |a-uau|$, we obtain
 \begin{align}\label{t_211_3}
(uv_mu)|a-uau|(uv_mu)^*q_m\geq \frac{\sqrt{3}}{2}\cdot(|a-z_0\textbf{1}|+u|a-z_0\textbf{1}|u)q_m.
 \end{align}
 To complete the proof, it remains to define
 $$w=e+\sum_{n=1}^{\infty}(v_n+uv_nu)$$
 which is a unitary (the series converges in the strong operator topology) (note here that $uv_mu\in \calU(q_m\calM q_m)$).
 We observe that
 \begin{align}
 	w|[a,u]|w^*p_n = w|[a,u]|p_nv_n^*p_n = wp_n|[a,u]|v_n^*p_n = v_n|a-uau|v_n^*p_n
 \end{align}
and similarly, $w|[a,u]|w^* q_n = (uv_nu)|a-uau|(uv_nu)^*q_n$ and $w|[a,u]|w^*e = |[a,u]|e =0$. Summing up the inequalities (\ref{t_211_2}) and (\ref{t_211_3}) in the measure topology we arrive at 
\begin{align*}
	w|[a,u]|w^* &= w|[a,u]|w^*e + \sum_{n=1}^{\infty}w|[a,u]|w^* (p_n + q_n)  \\
	&= \sum_{n=1}^{\infty}v_n|a-uau|v_n^*p_n + (uv_nu)|a-uau|(uv_nu)^*q_n\\
	&\geq \sum_{n=1}^{\infty}\frac{\sqrt{3}}{2}(|a-z_0\mathbf{1}| + u|a-z_0\mathbf{1}|u)(p_n+q_n)\\
	&= \frac{\sqrt{3}}{2}(|a-z_0\mathbf{1}| + u|a-z_0\mathbf{1}|u)
\end{align*}
which proves \eqref{t_211_1}.
Regarding the proof of equality \eqref{t_211_11}, see the end of the proof of the Theorem \ref{I_n}.

\end{proof}
	We have now established in \cref{I_n} and \cref{t_211} that for finite factors the commutator estimate \eqref{bs_t1_2_2} holds with the constant $\frac{\sqrt{3}}{2}$. However, this may not be the best constant for which,  for all normal $a\in \calM$, the inequality holds. We will now establish upper bounds on the best possible constant and we will in particular show that $\frac{\sqrt{3}}{2}$ is in fact the best possible constant when $\calM$ is a II$_1$-factor or a I$_n$-factor ($n<\infty$) with $n\equiv 0\mod 3$, .
	To do this we need the following proposition, which is partly motivated by the proof of \cite[Theorem 1]{HoWi}. Here, for a given algebra $\calA$ we denote by $\Mat_n(\calA)$  the set of all $n\times n$ matrices with entries in $\calA$.

\begin{proposition}\label{prop:restrict-to-permutation-matrix}
    Let $\calN$ be a finite factor with a faithful normal tracial state $\tau_{\calN}$, $\mathbb{M}_n=\Mat_n(\CC),\ n\in\NN$,  $\calM=\mathbb{M}_n\otimes\calN\cong \Mat_n(\calN)$, $\tau_{\calM}=\frac{1}{n}Tr\otimes\tau_{\calN}$ be a tracial state on $\calM$.
    Denote $\calU_n^{per}\subseteq \mathbb{M}_n$ for the group of permutation matrices and $\mathbb{D}_n\subseteq \mathbb{M}_n$ for the set of diagonal matrices.

    If $a\in \mathbb{D}_n\otimes\textbf{1}_{\calN}$ then
    \begin{align*}
        \sup_{u\in \calU(\calM)}\|a - u^*au\|_{2} = \max_{u \in \calU_n^{per}\otimes\textbf{1}_{\calN}}\|a - u^*au\|_{2}
    \end{align*}
    (The isomorphism (identification) of $\Mat_n(\calN)\to \mathbb{M}_n\otimes\calN$ is given by the mapping $(a_{ij})_{i,j=1}^n\to\sum_{i,j=1}^n \alpha_{ij}\otimes a_{ij}$ where $\alpha_{ij}$ are matrix units of $\mathbb{M}_n$.)
  
\end{proposition}
\begin{proof}
    Write $a = \Diag(a_{i})_{i=1}^n\otimes\textbf{1}_{\calN}$ with $a_i\in \CC$ and let $u=(u_{ij})_{i,j=1}^n\in \calU(\calM),\ u_{ij}\in\calN,\ i,j=1,\dots,n$.
    We note that $$\|a -uau^*\|_2^2 = \tau_{\calM}((a - uau^*)(a^*-ua^*u)) = 2\tau_{\calM}(|a|^2) - 2\Re(\tau_{\calM}(aua^*u^*)).$$
    We are interested in finding a unitary element $u\in \calM$ for which the scalar $$R(u) := -\Re(\tau_{\calM}(aua^*u^*)) = -\frac{1}{n}\sum_{i,j} \Re(\tau_{\calN}(a_{i}u_{ij}\overline{a_{j}}u_{ij}^*)) =  -\frac{1}{n}\sum_{i,j} \Re(a_{i}\overline{a_{j}})\tau_{\calN}(u_{ij}u_{ij}^*)$$ attains its maximum. For convenience,  let $(d_{ij})\in \mathbb{M}_n$ be the matrix given by $d_{ij} = -\frac{1}{n}\Re(a_{i}\overline{a_{j}})$, so that $R(u) = \sum_{ij}d_{ij}\tau_{\calN}(u_{ij}u_{ij}^*)$.
    Denote
    $\calW_n = \{(\tau_{\calN}(v_{ij}v_{ij}^*))_{ij} \in \mathbb{M}_n:\ v=(v_{ij})\in \calU(\Mat_n(\calN))\}$.
    We observe for $w = (\tau_{\calN}(v_{ij}v_{ij}^*))_{ij}\in \calW_{n}$  and every $j$ such that $1\leq j\leq n$, we have $\sum_{i}w_{ij} = \tau_{\calN}(\sum_i v_{ij}v_{ij}^*) = \tau_{\calN}(\mathbf{1}_{\calN}) = 1$. Similarly, for every $1\leq i\leq n$ we have that $\sum_{j}w_{ij} = \tau_{\calN}(\sum_j v_{ij}v_{ij}^*) = \tau_{\calN}(\mathbf{1}_{\calN}) = 1$. Furthermore, as $v_{ij}v_{ij}^*\geq 0$ in $\calN$, it is clear that $w_{ij}\geq 0$ for all $i,j$.
    Now, denote by $\calX_n$ the set of all elements $x = (x_{ij})\in \mathbb{M}_n$ satisfying
    \begin{align*}
        \forall j: \sum_{i}x_{ij} =1, \quad \forall i: \sum_{j}x_{ij} = 1, \quad \forall i,j: x_{ij}\geq 0 \quad
    \end{align*}
    so that $\calW_n\subseteq \calX_n$.  Considering $\calX_n$ as a subset of $\RR^{n^2}$, we see that $\calX_n$ defines a closed convex polytope. By \cite[Lemma]{HoWi}, the vertices of $\calX_n$ are the permutation matrices. Hence the maximum of the linear form $(x_{ij})\to \sum_{ij}d_{ij}x_{ij}$ on $\calX_n$ is attained for some permutation matrix $\widetilde{u} = (\widetilde{u}_{ij})\in \calU_n^{per}$. As $\widetilde{u}\in \calU_{n}^{per}\subseteq \Mat_n(\calN)$ we have that $\tau_{\calN}(\widetilde{u}_{ij}\widetilde{u}_{ij}^*) = \widetilde{u}_{ij}$ and so
    $$R(\widetilde{u}) = \sum_{i,j}d_{ij}\tau_{\calN}(\widetilde{u}_{ij}\widetilde{u}_{ij}^*) =
    \sum_{i,j}d_{ij}\widetilde{u}_{ij} = \max_{x\in \calX_{n}}\sum_{i,j}d_{ij}x_{ij}\geq \sup_{w\in \calW_n}\sum_{i,j}d_{ij}w_{ij}= \sup_{\substack{u\in \calU(\calM)}}R(u).$$
Thus,  $\sup_{\substack{u\in \calU(\calM)}}\|a - u^*au\|_{2} \leq \|a - (\widetilde{u}\otimes\textbf{1}_\calN)^*a(\widetilde{u}\otimes\textbf{1}_\calN)\|_{2}$ and the claim follows.	
\end{proof}

Combining \cref{I_n} and \cref{t_211}, we estimate  the maximal constant $C_{\calM}$ that satisfies the commutator estimate \eqref{bs_t1_2_2} for finite factors $\calM$ in  \cref{t3} below. For the definitions of the constants  $\Lambda_n$ and $\widetilde{\Lambda_n}$ we refer to \eqref{eq:definition_lambda_n} and \eqref{eq:definition-tilde-lambda_n}  and for the exact values of $\Lambda_n$ we refer to \cref{t_techmain}.

\begin{theorem}\label{t3}
	Let $\calM$ be a finite factor with $\calM\not=\CC$. Then there is a constant $C>0$ with the property that:
	\begin{itemize}
		
	\item [($*$)] For any normal $a\in S(\calM)$ there exists a complex number $z_0\in \CC$ and unitaries $u,v,w\in \calU(\calM)$ such that \begin{align}\label{eq:optimal-constant-commutator}|[a,u]| \geq C(v|a-z_0\mathbf{1}|v^* + w|a-z_0\mathbf{1}|w^*).\end{align}
	
\end{itemize}
Moreover, a maximal constant $C_{\calM}$ with this property exists and it satisfies $\Lambda_n\leq C_{\calM}\leq \frac{1}{2}\widetilde{\Lambda}_n$ when $\calM$ is a I$_n$-factor ($1<n<\infty$), and $C_{\calM}$ equals $\frac{1}{2}\sqrt{3}$ when $\calM$ is a II$_1$-factor.
\end{theorem}
\begin{proof}
	Combining \cref{I_n} and \cref{t_211} we obtain for any finite factor that the constant $C=\frac{1}{2}\sqrt{3}$ is admissable for ($*$). By \cref{t_techmain} we have that $\Lambda_n =\frac{1}{2}\sqrt{3}$ when $n=3$ or $5\leq n\leq \infty$. Let $n<\infty$. To see that $C=\Lambda_n$ is admissible for all $n$ we note that by \cref{theorem-transformation-with-bound} we have for $g\in S(\Omega_n)$ that there exist $z_0\in \CC$, $T\in \Aut_n$ such that $\Lambda(g,T,z_0)=\Lambda_n(g)\geq \Lambda_n$, which means
	\begin{align}\label{eq:constant-C_M-function-inequality}
		|g\circ T - g| \geq \Lambda_n(|g-z_0| + |g\circ T-z_0|).
	\end{align} Repeating the proof of \cref{I_n}, replacing \eqref{eq:I_n-function-inequality} with \eqref{eq:constant-C_M-function-inequality}, we obtain that  $C=\Lambda_n$ is also an admissible constant for ($*$).
	We will later see that the maximal admissible constant $C_{\calM}$ actually exists. First we prove upper bounds on constants $C$ satisfying ($*$) for $\calM$. Let $\tau$ be a tracial state on $\calM$. 
	
	Let $\calM$ be a I$_n$-factor with $1<n<\infty$. Let  $g\in S(\Omega_n)$ be the the function from \cref{lemma:optimality-triangle-function} and let $a = \Diag(g(1),...g(n))\in \calM$. Let $z_0\in \CC$, $u,v,w\in \calU(\calM)$ such that ($*$) is satisfied for $a$ with constant $C$. It follows from  \cref{prop:restrict-to-permutation-matrix}  ($\calN=\CC$) that
$$\|[a,u]\|_1\leq\|[a,u]\|_2=\|a - u^*au\|_2\leq \max_{u_0 \in \calU_n^{per}}\|a - u_0^*au_0\|_2\leq\Diam(\sigma(a)).$$ 
Hence,
$$2C\|a-z_0\textbf{1}\|_1= C\|v|a-z_0\textbf{1}|v^*+w|a-z_0\textbf{1}|w^*\|_1\leq\|[a,u]\|_1\leq\Diam(\sigma(a)).$$
Now, choosing $g$ as in the assertion of \cref{lemma:optimality-triangle-function} we obtain
$$1\geq \Diam(\sigma(a))\geq 2C\|a-z_0\textbf{1}\|_1\geq 2C\|g-z_0\|_1\geq 2C\widetilde{\Lambda}_n^{-1}.$$
 Hence, $C\leq \frac{1}{2}\widetilde{\Lambda}_{n}$.

Let $\calM$ be of type II$_1$. Then $\calM\cong \Mat_3(\CC)\otimes\calN$  for some II$_1$-factor $\calN$. Let the function $g\in S(\Omega_3)$ be as in \cref{lemma:optimality-triangle-function} and let $a_1 =\Diag(g(1),g(2),g(3))\in \Mat_3(\CC)$ and $a=a_1\otimes\textbf{1}_\calN\in \calM$. Let $z_0\in \CC$, $u,v,w\in \calU(\calM)$ be such that ($*$) holds for $a$ with constant $C$.
We have 
$$\|[a,u]\|_1\leq\|[a,u]\|_2=\|a - u^*au\|_2\leq \max_{u_0 \in \calU_3^{per}\otimes\textbf{1}_\calN}\|a - u_0^*au_0\|_2=\max_{u_0 \in \calU_3^{per}}\|a_1 - u_0^*a_1u_0\|_2\leq\Diam(\sigma(a_1)).$$
Hence, $$2C\|a_1-z_0\textbf{1}\|_1=2C\|a-z_0\textbf{1}\|_1= C\|v|a-z_0\textbf{1}|v^*+w|a-z_0\textbf{1}|w^*\|_1\leq\|[a,u]\|_1\leq\Diam(\sigma(a_1))\leq 1.$$
It follows from \cref{lemma:optimality-triangle-function} that
$$1\geq \Diam(\sigma(a_1))\geq 2C\|g-z_0\|_1\geq 2C\widetilde{\Lambda}_3^{-1}.$$
Hence, $C\leq \frac{1}{2}\widetilde{\Lambda}_{3}=\frac{\sqrt{3}}{2}$. For $\calM$ a II$_1$-factor, this shows that in fact $C_{\calM}$ exists and that $C_{\calM} =\frac{1}{2}\sqrt{3}$. 

We now show that the maximal constant $C_{\calM}$ also exists when $\calM$ is a I$_n$-factor ($1<n<\infty$). 	Let $(C_i)_{i\geq 1}$ be an increasing sequence of positive constants admissible for ($*$) and set $C=\sup C_i\leq \frac{1}{2}\widetilde{\Lambda}_n$. For a normal $a\in \calM$ there exists corresponding $u_i\in \calU(\calM)$ and $z_{0,i}\in \CC$ such that the equation \eqref{eq:constant-C_M-function-inequality} holds with the constant $C_i$.
Now by  
	$$2\|a\|_1\geq \|[a,u_i]\|_1 \geq 2C_i\|a-z_{0,i}\mathbf{1}\|_1 \geq 2C_i(|z_{0,i}| - \|a\|_1)$$ 
	we obtain $|z_{0,i}|\leq \frac{1+C_i}{C_i}\|a\|_1\leq \frac{1+C_1}{C_1}\|a\|_1$. Therefore, as the sequences $(u_i)_{i}$ and $(z_{0,i})_{i}$ are bounded and as $\calM$ is finite-dimensional, we can assume these sequences converge in norm to some $u\in \calU(\calM)$ and some $z_{0}\in\CC$ (otherwise restrict to a subsequence). Now the elements
$d_i:= |[a,u_i]|-C_i(|a-z_{0,i}\mathbf{1}| + u_i|a-z_{0,i}\mathbf{1}|u_i^*)$ are all positive and converge to $d= |[a,u]|-C(|a-z_0\mathbf{1}| + u|a-z_0\mathbf{1}|u^*)$. As the cone of positive elements in $\calM$ is closed in the norm, we obtain $d\geq 0$. This shows that $|[a,u]|\geq C(|a-z_0\mathbf{1}| + u|a-z_0\mathbf{1}|u^*)$ holds, and therefore $C$ is admissible for ($*$) as well. Hence, the supremum of all admissible constants (which is finite), is again admissible, and this shows that $C_{\calM}$ exists. It now follows that $\Lambda_n \leq C_{\calM}\leq \frac{1}{2}\widetilde{\Lambda}_{n}$

\end{proof}

\section{Commutator estimates for normal operators in infinite factors}\label{section:estimates-in-infinite-factors}

We shall now obtain the commutator estimate \eqref{bs_t1_2_2} for normal elements in an infinite factor. We show in \cref{t_inf} that for such factors the constant $C$ in this estimate can be chosen arbitrary close to $1$. For infinite factors, this extends the result of \cite[Theorem B.1]{BBS}  to normal elements. The proof of \cref{t_inf} extensively uses the geometry of projections. Before we start its proof, we state and prove three short technical lemmas. 
\begin{lemma}\label{l1_inf}
    Let $\calM$ be an infinite factor and $p$ be a infinite projection from $\calM$. If $p_1,\dots,p_n\in \calP(\calM)$ are pairwise commuting and $p_1,\dots,p_n\prec p$, then $p_1\vee\dots\vee p_n\prec p$.
    (We understand the symbol ``$\prec$'' as ``$\precnsim$''.)
\end{lemma}
\begin{proof} Let $q_1=p_1$ and $q_{k+1}=p_{k+1}(\textbf{1}-q_1-\dots q_k)$ for  $k=1,\dots,n-1$.
    Then $q_iq_j=0$ for $i\neq j$, $q_k\prec p$ for $k=1,\dots,n$ and $p_1\vee\dots\vee p_n=q_1+\dots+q_n\prec p$ (see \cite[Lemma 2 (ii)]{BS2012}).
\end{proof}
\begin{lemma}\label{l2_inf}
    Let $\calM$ be a factor, $a$ be a normal operator from $S(\calM)$, $p,q\in \calP(\calM),\ q\preceq p$. Suppose that one of the
    following conditions holds:

    (i). $q$ is finite and there exists a sequence of finite projections $(p_n)$ in $\calM$ such
    that $p_n\uparrow p$ and $[a,p_n] = 0$ for all $n\in\mathbb{N}$;

    (ii). $q$ is an infinite projection and $[a,p]=0$.

    Then there exists a projection $q_1\in\calM$ such that $q_1\sim q,\ [a,q_1] = 0$ and such that $q_1\leq p$.
\end{lemma}
\begin{proof}
    The proof follows along the lines  of \cite[Lemma 3]{BS2012} and is therefore omitted.
\end{proof}

\begin{lemma}\label{l3_inf}
    Let $\calM$ be a von Neumann algebra, $a,b\in LS(\calM)$, $\alpha_1,\alpha_2>0$, and
    $$|a|\geq \alpha_1\textbf{1},\ ,\ 2\alpha_2<\alpha_1,\ \alpha_2\textbf{1}\geq |b|.$$
    Then there exists  $v\in \calU(\calM)$ such that
    $$v|a-b|v^*\geq(1-\frac{2\alpha_2}{\alpha_1})|a|+|b|.$$
\end{lemma}
\begin{proof}
    Let $a,b\in LS(\calM)$, $\alpha_1,\alpha_2>0$ satisfy the assumption of the lemma. By \cref{t_AAP_21}, we have that
    $$|a|\leq v|a-b|v^*+w|b|w^*$$
    for some $v,w\in\calM$ with  $v^*v=w^*w=\textbf{1}$.
    Then
    \begin{align*}
        v|a-b|v^*&\geq|a|-w|b|w^*\geq |a|-\alpha_2ww^*\geq |a|-\alpha_2\textbf{1}\\
        &\geq |a|+|b|-2\alpha_2\textbf{1}\geq |a|+|b| - \frac{2\alpha_2}{\alpha_1}|a|=(1-\frac{2\alpha_2}{\alpha_1})|a|+|b|.
    \end{align*}
    Since $v|a-b|v^*\geq (1-\frac{2\alpha_2}{\alpha_1})|a|\geq(\alpha_1-2\alpha_2)\textbf{1}$, it follows
    $$0=(\textbf{1}-vv^*)v|a-b|v^*(\textbf{1}-vv^*)\geq(\alpha_1-2\alpha_2)(\textbf{1}-vv^*)\geq 0.$$
    Therefore, we have $\textbf{1}-vv^*=0$, i.e. $v\in \calU(\calM)$.
\end{proof}


\begin{theorem}\label{t_inf}
    Let $\calM$ be an infinite factor, and let $a\in S(\calM)$ be normal. There is a $\lambda_0\in\mathbb{C}$ such that for any $\varepsilon>0$ there exist $u_\varepsilon=u_\varepsilon^*\in \calU(\calM)$, $w_\varepsilon\in \calU(\calM)$ so that
    \begin{align}\label{t_inf_1}
        w_\varepsilon|[a,u_\varepsilon]|w_\varepsilon^*\geq (1-\varepsilon)(|a-\lambda_0\textbf{1}|+u_\varepsilon|a-\lambda_0\textbf{1}|u_\varepsilon).
    \end{align}
\end{theorem}
\begin{proof}
Let $e(\cdot)$ be the spectral measure of $a$ on $\mathbb{C}$,  in particular, $e(X)=\chi_X(a)$ for any $X\in\calB(\mathbb{C})$. Since $a\in S(\calM)$ there exists a $R>0$ so that $e(X_R)$ is a finite projection, where $X_R=\{\lambda\in\mathbb{C}:\ |\lambda|>R\}$. Then $Y_R:=\mathbb{C}\setminus X_R$ is compact and it follows from Lemma \ref{l1_inf} that $e(Y_R)\sim\textbf{1}$.
	A point $\lambda\in\mathbb{C}$ will be called a \textit{point of densification} for $a$ if $e(V)\sim\textbf{1}$ for any neighborhood $V$ of a point $\lambda$. Denote by $A$ the set of all points of densification for $a$. 
    
    We claim that $A\neq\emptyset$. To see that the claim holds it is sufficient to show there exists a system of nested sets $B_n=[\alpha_n,\alpha_n+\frac{5R}{2^n})\times[\beta_n,\beta_n+\frac{5R}{2^n})$, with $e(B_n)\sim\textbf{1}$.
	We put $\alpha_1 = \beta_1 = -R$ so that clearly $Y_R\subset B_1$ and therefore $e(B_1)\sim\textbf{1}$. Now suppose $\alpha_1,\beta_1,\dots,\alpha_n,\beta_n$ are already constructed so that $e(B_1)\sim\dots\sim e(B_n)\sim\textbf{1}$. We can divide the rectangle $B_n$ into $4$ smaller rectangles by $$B_n=\bigcup_{k,l=0}^1 [\alpha_n+k\cdot\frac{5R}{2^{n+1}},\alpha_n+(k+1)\cdot\frac{5R}{2^{n+1}})\times[\beta_n+l\cdot\frac{5R}{2^{n+1}},\beta_n+(l+1)\cdot\frac{5R}{2^{n+1}}).$$ It follows from Lemma \ref{l1_inf} that one of the sets from this union can be taken for $B_{n+1}$ (which then defines $\alpha_{n+1},\beta_{n+1}$). This completes the induction. The point $\lambda:=(\sup_n \alpha_n) + (\sup_n\beta_n)i$ is a point of densification for $a$ since any neighbourhood $V$ of $\lambda$ contains a set $B_n$ for some $n$. Therefore $A\neq\emptyset$.
	
	We show that $A$ is closed. Indeed, if $\lambda$ is a limit point of $A$ and $V$ is a neighborhood of $\lambda$, then $V$ is also a neighborhood of some point from $A$. Hence $e(V)\sim \textbf{1}$. This shows $\lambda\in A$. Thus $A$ is closed.
	Obviously, $A\subset Y_R$. Therefore, $A$ is a nonempty compact subset in $\mathbb{C}$.

    Let us consider three cases covering the full picture.
    \begin{itemize}
    	\item  1. \textit{There is a point $\lambda_0\in\mathbb{C}$ such that $e(\{\lambda_0\})\sim\textbf{1}$.} Then $e(\mathbb{C}\setminus\{\lambda_{0}\})\preceq e(\{\lambda_0\})$ and therefore there is a $v\in\calM$ with $v^*v = e(\mathbb{C}\setminus\{\lambda_{0}\})$ and $vv^* \leq e(\{\lambda_0\})$.
   		Let 's put $u=v+v^*+(e(\{\lambda_0\})-vv^*)$. Then $u=u^*\in \calU(\calM)$. Since
    	\begin{align*}
    		(a-\lambda_0\textbf{1})u(a-\lambda_0\textbf{1})^*&=(a-\lambda_0\textbf{1})u(\textbf{1}-e(\{\lambda_0\}))(a-\lambda_0\textbf{1})^*\\
    		&=(a-\lambda_0\textbf{1})v(\textbf{1}-vv^*)(a-\lambda_0\textbf{1})^*\\
    		&=(a-\lambda_0\textbf{1})e(\{\lambda_0\})u(a-\lambda_0\textbf{1})^*=0
    	\end{align*}
    	and, similarly, $$(a-\lambda_0\textbf{1})^*u(a-\lambda_0\textbf{1})=0$$
    	then
    	$$|[a,u]|=|(a-\lambda_0\textbf{1})-u(a-\lambda_0\textbf{1})u|=|a-\lambda_0\textbf{1}|+u|a-\lambda_0\textbf{1}|u$$
    	which shows the result for this case with $w_{\varepsilon}=\textbf{1}$.
    \end{itemize}

    In the following two cases, the scalar $\lambda_0\in\mathbb{C}$ will be found and for a fixed number $\varepsilon>0$ a sequence of pairs of projectors $\{(p_n,q_n)\}_{n\geq 1}$ of $\calM$ will be constructed together with a sequence $(\gamma_{n}$) of positive numbers converging to zero satisfying the following conditions:

    (i). $p_nq_m=0,\ p_np_m=\delta_{nm}p_n,\ q_nq_m=\delta_{nm}q_n,\ [a,p_n]=[a,q_n]=0,\ p_n\sim q_n$ for all $n,m$;

    (ii). $q_n\leq e(W_n),\ p_n\leq e(V_n)$ for all $n\geq 1$;

    (iii). $\bigvee_{n\geq 0}p_n\vee\bigvee_{n\geq 0}q_n=\textbf{1}-e(\{\lambda_0\})$,\\
   where $V_n := \{\lambda: |\lambda-\lambda_0|>\gamma_n\}$ and  $W_n := \{\lambda: |\lambda-\lambda_0| < \frac{\varepsilon}{2}\gamma_n\}$.

    \begin{itemize}
    	\item 2. \textit{The set $A$ has a limit point $\lambda_0$.}
    	 We can assume that $\varepsilon<\frac{1}{2}$.
    	We inductively construct the sequences of positive numbers $(\gamma_n)$ (and hence the sets $V_n$, $W_n$), numbers $(\lambda_n)$ from $A$, and sets
    	\begin{align}
    	 U_n=\{\lambda:\ |\lambda-\lambda_{2n}|<\gamma_{n+1}\}
    	\end{align}
    	in such a way that $U_{n}\subseteq W_n\cap V_{n+1}$ and the set $V_{n+1}\setminus\bigcup_{k=1}^{n}(U_k\cup V_k))$ is a neighborhood of the point $\lambda_{2n+1}$.
    	First, let $\lambda_1\in A\setminus\{\lambda_{0}\}$ and put $\gamma_1=\frac{|\lambda_1-\lambda_0|}{2}$. Then $V_1$ is a neighborhood of the point $\lambda_1$.
    	Next, in the set $W_1$ there will be different points $\lambda_2,\lambda_3$ from $A\setminus\{\lambda_{0}\}$. Put $\gamma_2=\frac{1}{2}\min\{|\lambda_3-\lambda_0|,|\lambda_{2}-\lambda_3|, \frac{\varepsilon}{2}\gamma_1-|\lambda_2-\lambda_0|,|\lambda_2-\lambda_0|\}$ and note that $\gamma_2<\frac{1}{2}|\lambda_3-\lambda_0|\leq \frac{\gamma_1}{4}$.
    	Note also that the set $V_2\setminus(V_1\cup U_1)$ is a neighborhood of the point $\lambda_3$ and that $U_1\subseteq W_1\cap V_2$.
    	We continue this process by induction. Let these sequences be constructed for the indices $1,\dots,n$. Then in the set $W_n$ there will be different points $\lambda_{2n},\lambda_{2n+1}$ from $A\setminus\{\lambda_{0}\}$. Put  $\gamma_{n+1}=\frac{1}{2}\min\{|\lambda_{2n+1}-\lambda_0|,|\lambda_{2n}-\lambda_{2n+1}|,\frac{\varepsilon}{2}\gamma_n - |\lambda_{2n}-\lambda_0|,|\lambda_{2n}-\lambda_0|\}$. Then $\gamma_{n+1}<\frac{\gamma_{n}}{4}$ and
    	$V_{n+1}\setminus\bigcup_{k=1}^{n}(U_k\cup V_k))$ is a neighborhood of the point $\lambda_{2n+1}$, and $U_n\subseteq W_n\cap V_{n+1}$. Thus, the above sequences are constructed. We remark that for $n<m$ we have $U_n\cap U_m\subseteq W_m\cap V_{n+1}=\emptyset$
    	
    	Put $p_1=e(V_1),\ q_1=e(U_1);\ q_n=e(U_n),\ p_n=e(V_n\setminus\bigcup_{k=1}^{n-1}(U_k\cup V_k))$ for $n>1$. Then we have by the construction that $p_1,q_1,p_2,q_2,...$ are pairwise orthogonal and $p_n\sim\textbf{1}\sim q_n$ for any $n$.
    	Now since $V_n=(V_n\setminus\bigcup_{k=1}^{n-1}(U_k\cup V_k)))\cup\bigcup_{k=1}^{n-1}(U_k\cup V_k))$ and $\bigcup_{n=1}^\infty V_n=\mathbb{C}\setminus\{\lambda_0\}$ we find $\bigvee_{n\geq 0}p_n\vee\bigvee_{n\geq 0}q_n=\textbf{1}-e(\{\lambda_0\})$.
    	\item   3. \textit{The set $A$ is finite and $e(\{\lambda\})\prec \textbf{1}$ for any $\lambda\in A$.} We can by assumption write $A = \{\lambda_0,\dots, \lambda_m\}$ for some $m\geq 0$ (note $A$ is non-empty). When $|A|=1$ put $r=1$ and when $|A|>1$ let $r$ be the minimum distance between points in $A$.    	
    	Consider the sets $V(t)=\mathbb{C}\setminus\bigcup_{k=0}^m\{\lambda:\ |\lambda-\lambda_k|<t\}$ for $0<t<\frac{r}{2}$.
    	It is clear that $V(t)\uparrow\mathbb{C}\setminus A$ at $t\downarrow 0$.
    	
    	We show that $e(V(t))\prec\textbf{1}$ for $0<t<\frac{r}{2}$.
    	Indeed, for any point $z\in V(t)\setminus X_R$ there is a neighborhood $U_z$ of $z$ with $e(U_z)\prec\textbf{1}$. Now as the set $V(t)\setminus X_R$ is compact we can let $\{U_{z_1},\dots,U_{z_l}\}$ be a finite subcover for $V(t)\setminus X_R$.
    	Then $\{X_R,U_{z_1},\dots,U_{z_l}\}$ is the coverage of the set $V(t)$. It follows from Lemma  \ref{l1_inf} that $e(V(t))\prec\textbf{1}$.
    	
    	There are now two possible cases:
    	
    	3.1. All projections $e(V(t)),\ t>0$, are finite. In this case, put $\gamma_1=\frac{r}{3}$.
    	
    	3.2. There is a $0<t_0<\frac{r}{3}$ so that the projection $e(V(t_0))$ is infinite. In this case put $\gamma_1=t_0$.
    	
    	Set $\gamma_{n}=\frac{\gamma_1}{2^{n-1}},\ n>1$ (and hence $V_n,W_n$ are defined as well); We set $p_1:=e(V(\gamma_1)\cup (A\setminus\{\lambda_0\}))\leq e(V_1)$.  It follows from Lemma \ref{l1_inf} that $p_1\prec \textbf{1}$ and $p:=e(W_1)\sim \textbf{1}$.
    	If we put $q=p_1$, then for $p,q$ the conditions Lemma \ref{l2_inf} are met: condition (ii) is met if $q$ is an infinite projection, and condition (i) is met in case 3.1 if $q$ is an finite projection (in this case, the set $W_1$ is covered by the system $V(t),\ t>0$). Therefore, there is a projection $q_1\leq e(W_1)$ such that $q_1\sim p_1$ and $[a,q_1]=0$.
    	Now, suppose the projections $p_1,q_1,\dots,p_n,q_n\prec\textbf{1}$ are constructed. We build projections $p_{n+1},q_{n+1}$. We put
    	 $p_{n+1}=e(V(\gamma_{n+1}))\cdot(\textbf{1}-\sum_{k=1}^n(p_k+q_k))$. Then $p_{n+1}\prec\textbf{1}$ since $p_{n+1}\leq e(V(\gamma_{n+1}))$.
    Furthermore, since $e(W_n)\sim \textbf{1}$ and $p_1,q_1,\dots,p_n,q_n\prec\textbf{1}$ we find $e(W_n)\cdot(\textbf{1}-\sum_{k=1}^n(p_k+q_k))\sim\textbf{1}$. Again using Lemma \ref{l2_inf}, we find such a projection $q_{n+1}\sim p_{n+1}$ that $q_{n+1}\leq e(W_n)\cdot(\textbf{1}-\sum_{k=1}^n(p_k+q_k))$ and $[a,q_{n+1}]=0$ (two cases are considered again: $p_{n+1}$ is a infinite projection; $p_{n+1}$ is a finite projection and the condition 3.1 is met). As $p_{n+1}+\sum_{k=1}^n(p_k+q_k)\geq e(V(\gamma_{n+1}))$ and $p_1\geq e(A\setminus \{\lambda_0\})$ we conclude $\sum_{k=1}^\infty(p_k+q_k)=\textbf{1}-e(\{\lambda_0\})$.
    	Therefore, the projections $p_1,q_1,p_2,q_2,\dots$ satisfy the conditions (i)-(iii).
    \end{itemize}
 	In the cases (2) and (3) we can now find partial isometries $v_n\in\calM$ so that $v_n^*v_n=p_n,\ v_nv_n^*=q_n,$ for  $n=1,2,\dots$. We put $u_\varepsilon=e(\{\lambda_0\})+\sum_{n=1}^\infty(v_n+v_n^*)$. Then $u_\varepsilon=u_\varepsilon^*\in \calU(\calM)$, $u_\varepsilon e(\{\lambda_0\})=e(\{\lambda_0\})$ and $u_\varepsilon p_n=q_nu_\varepsilon$ for all $n$.
    We have
    \begin{align}\label{t1_1}
        |a-\lambda_0\textbf{1}|p_n\geq\gamma_{n}p_n,\ |a-\lambda_0\textbf{1}|q_n\leq \frac{\varepsilon}{2}\gamma_nq_n,\ \forall n.
    \end{align}
    Therefore
    \begin{align}\label{t1_2}
        |u_\varepsilon au_\varepsilon-\lambda_0\textbf{1}|p_n=u_\varepsilon|a-\lambda_0\textbf{1}|q_nu_\varepsilon\leq \frac{\varepsilon}{2}\gamma_nu_\varepsilon q_nu_\varepsilon=\frac{\varepsilon}{2}\gamma_np_n,\ \forall n.
    \end{align}

    Since $[a,p_n]=[u_\varepsilon au_\varepsilon,p_n]=0$ then
    \begin{align}\label{t1_3}
        |a-u_\varepsilon au_\varepsilon|p_n=|(a-\lambda_0\textbf{1})p_n-(u_\varepsilon au_\varepsilon-\lambda_0\textbf{1})p_n|.
    \end{align}

    It follows from \cref{l3_inf} that
    $$w_n|a-u_\varepsilon au_\varepsilon|p_nw_n^*\geq ((1-\varepsilon)|a-\lambda_0\textbf{1}|+|u_\varepsilon au_\varepsilon-\lambda_0\textbf{1}|)p_n$$
    for some $w_n\in \calU(p_n\calM p_n)$.

    Therefore
    \begin{align}\label{t1_4}
        w_n|a-u_\varepsilon au_\varepsilon|w_n^*p_n\geq((1-\varepsilon)(|a-\lambda_0\textbf{1}|+|u_\varepsilon au_\varepsilon-\lambda_0\textbf{1}|))p_n.
    \end{align}
    Applying the automorphism $u_\varepsilon\cdot u_\varepsilon$ to (\ref{t1_4}), and noting that $u_{\varepsilon}|a-u_{\varepsilon}au_{\varepsilon}|u_{\varepsilon} = |a-u_{\varepsilon}au_{\varepsilon}|$, we obtain
    \begin{align}\label{t1_5}
        (u_\varepsilon w_nu_\varepsilon)|a-u_\varepsilon au_\varepsilon|(u_\varepsilon w_nu_\varepsilon)^*q_n\geq ((1-\varepsilon)(|a-\lambda_0\textbf{1}|+|u_\varepsilon au_\varepsilon-\lambda_0\textbf{1}|))q_n.
    \end{align}
	Recall that $S(\calM)=\calM$ if $\calM$ has type I or III. In this case, we denote by $t$ the strong operator topology in $\calM$. If the factor $\calM$ is of type II then $\calS(\calM)=\calS(\calM,\tau)$ for any faithful semi-finite normal trace $\tau$ on $\calM$. In this case we let $t$ stand for the measure topology $t_{\tau}$ (this topology is defined in \cref{section:preliminaries}, the need to use this topology is due to the fact that $a$ can be an unbounded operator). 
	
    To complete the proof, it remains to set
    $$w_\varepsilon=e(\{\lambda_0\})+\sum_{n=1}^{\infty}(w_n+u_\varepsilon w_nu_\varepsilon)$$ (the series converges in the strong operator topology)
    and sum up the inequalities (\ref{t1_4}) and (\ref{t1_5}) in the topology $t$.
\end{proof}

\section{Estimates for inner derivations associated to normal elements}\label{section:estimates-for-derivations}
In this section we apply the operator estimates from \cref{I_n} and \cref{t3} to extend the result of \cite[Theorem 1.1]{BBS} and estimate  the norm of inner derivations $\delta_{a}:\calM\to L_1(\calM,\tau)$ in the case when $\calM$ a finite factor with faithful normal trace $\tau$ and 
$a\in L_1(\calM,\tau)$ is normal.

We establish some notation first. Let $\calM$ be a von Neumann algebra with predual $\calM_*$. The Banach space $\calM_*$ can be embedded into its double dual $(\calM_*)^{**} = \calM^*$. In this way we identify $\calM_*$ with the space of ultra weakly continuous linear functionals on $\calM$. The predual $\calM_*$ is a Banach $\calM$-bimodule with the bimodule actions given by:
\begin{align}\label{eq:M-bimodule}
	(a\cdot\omega)(x)=\omega(xa),\ (\omega\cdot a)(x)=\omega(ax),\ a,x\in\calM,\ \omega\in \calM_*.
\end{align}
If there is a faithful normal semi-finite trace $\tau$ on $\calM$, then the Banach $\calM$-bimodule $\calM_*$ is isomorphic to $L_1(\calM,\tau)$ (see e.g.
\cite[Lemma 2.12 and Theorem 2.13]{Tak2}).

A linear operator $\delta:\calM\to\calM_*$ is called a \textit{derivation} if
$$\delta(xy)=\delta(x)y+x\delta(y)$$
for all $x,y\in\calM$. 
For each $a\in\calM_*$ a derivation $\delta_a:\calM\to\calM_*$ can be defined by the equality $$\delta_a(x)=[a,x]=ax-xa$$ (using the $\calM$-bimodule structure as defined in \eqref{eq:M-bimodule}). Such derivations are called \textit{inner}.
In fact it holds true that any derivation $\delta:\calM\to\calM_*$ is inner.  Moreover, there exists $a\in\calM_*$ so that $\delta=\delta_{a}$ and $\|a\|_{\calM_*}\leq \|\delta\|_{\calM\to\calM_*}$ see \cite[Theorem 4.1]{Haagerup} and \cite[Corollary  C]{BGM}. We are interested in describing the norm of the derivations $\delta_{a}:\calM\to \calM_*$ for $a\in \calM_*$. Is it true that a distance formula similar to \eqref{eq:derivation-distance-formula} holds true? This question has been fully settled in \cite[Theorem 3.1]{BBS} for infinite factors. Moreover, in \cite{BBS} the following theorem was proved:

\begin{theorem}\cite[Theorem 1.1]{BBS}
	If $\calM$ is a von Neumann algebra with a faithful normal finite trace $\tau$ and $a=a^*\in L_1(\calM,\tau)$, then  there exists  $c_a=c_a^* \in L_1(\calM,\tau)\cap Z(S(\calM))$ such that
	\begin{align}\label{bbs2}
		\left\| \delta_a \right\|_{\calM\to L_1(\calM,\tau)} =2\left\|a-c_a\right\|_1= 2 \min_{z\in Z(S(\calM))}\left\| a-z \right\|_1
	\end{align}
	where $Z(S(\calM))$ stands for the center of the algebra of all measurable operators affiliated with $\calM$ 
\end{theorem}
We focus on the case that $\calM$ is finite. For brevity, we will denote the norm $\|\cdot\|_{\calM\to L_1(\calM,\tau)}$ by $\|\cdot\|_{\infty,1}$. For general $a\in L_1(\calM,\tau)$ we do not know the relationship between $\left\| \delta_a \right\|_{\infty,1}$ and $\inf \{ \left\| a-z \right\|_1 :z\in Z(S(\calM))$. In \cref{tderivationbound}, we shall give upper and lower estimates of this relation in the case when $\calM$ is a finite factor and $a$ is a normal operator. We will see a substantial difference with the case of inner derivations associated to self-adjoint elements. First we state \cref{trel1} which is related and is used in the proof of \cref{tderivationbound}. Recall that when $n\equiv 0\ (\mod 3)$ or $n=\infty$ we have $2\Lambda_n = \sqrt{3} = \widetilde{\Lambda}_n$ and that in addition,
$$\lim_{n\to\infty}\widetilde{\Lambda}_n=\sqrt{3},$$
and 
$$ 2\Lambda_n=\sqrt{3}\ \text{for}\ n=3,\text{ or } n\geq 5.$$

For convenience, we define for a finite factor $\calM$ the value
\begin{align}
	n(\calM) = \begin{cases}
		n & \calM \text{ is a I$_n$-factor} \\
		\infty & \calM \text{ is a II$_1$-factor}
	\end{cases}
\end{align}

\begin{theorem}\label{trel1}
    Let $\calM$ be a finite factor with a faithful tracial state $\tau$. Assume $\calM\not=\CC$.  Then
    \begin{enumerate}
        \item For every derivation $\delta_{a}:\calM\to L_1(\calM,\tau)$ with $a\in \calM$ normal, there is a normal $b\in \calM$ such that $\delta_a = \delta_b$ and  $\|\delta_b\|_{\infty,1}\geq 2\Lambda_{n(\calM)}\|b\|_1$.
        \item There exists a normal $a\in \calM$ for which the derivation $\delta_a:\calM\to L_1(\calM,\tau)$ is non-zero and such that for every  $b\in \calM$ with $\delta_a = \delta_b$ we have $\|\delta_b\|_{\infty,1}\leq \widetilde{\Lambda}_{n(\calM)}\|b\|_1$.
    \end{enumerate}
\begin{proof}
	
	(1) Let $a\in \calM$ be normal. By \cref{t3} there exist $u,w\in \calU(\calM)$, $z_0\in \CC$ satisfying the commutator estimate \eqref{eq:optimal-constant-commutator}, hence $\|\delta_{a}\|_{\infty,1} \geq \|\delta_{a}(u)\|_1 \geq 2\Lambda_{n(\calM)}\|a-z_0\mathbf{1}\|_1$.
	This shows the result since  $b:=a-z\textbf{1}$ is normal and $\delta = \delta_a = \delta_{b}$.
	
	(2)  Let $\calM$ be a finite factor. When $\calM$ is a I$_n$-factor, we set $m:=n$ and we can write $\calM = \Mat_m(\calN)$, with $\calN=\CC$. When $\calM$ is a II$_1$-factor we set $m=3$ and we can write $\calM = \Mat_m(\calN)$ for some II$_1$-factor $\calN$. We now let $g\in L_\infty(\Omega_m)$ be non-constant and let $a$ be the diagonal matrix $a = \Diag(g(1),\ldots,g(m))\otimes\textbf{1}_{\calN}\in\Mat_m(\CC)\otimes\calN=\calM$. Then $\delta_a:\calM\to L_1(\calM,\tau)$ is a non-zero derivation.
	To estimate $\|\delta_a\|_{\infty,1}$ we recall that the Russo{\color{red}--}Dye Theorem, \cite[Theorem 1]{RuDy}, asserts for a unital $\Cstar$-algebra that the closed unit ball equals the closed convex hull of all the unitaries. 
	Now, for $x\in \Conv(\calU(\calM))$ we can write $x = \sum_{i=1}^{N}c_iu_i$ with $N\in\NN$, 
	$u_i\in \calU(\calM)$ and $c_i\geq 0$ with $\sum_{i=1}^{N}c_i =1$. Then clearly $\|\delta_a(x)\|_1 \leq \sum_{i=1}^{N}c_i \|\delta_a(u_i)\|_1 \leq \max_{1\leq i\leq N}\|\delta(u_i)\|_{1} \leq \sup_{u\in \calU(\calM)}\|\delta_a(u)\|_1$. By continuity of $\delta_{a}$ this inequality holds for all $x$ in the closed convex hull as well. By the Russo-Dye Theorem this shows that 
	\begin{align}\label{eq:russo-dye-result}
		\|\delta_a\|_{\infty,1} = \sup_{x\in \calM, \|x\|\leq 1}\|\delta_a(x)\|_1 = \sup_{x\in\overline{ \Conv(\calU(\calM))}}\|\delta_a(x)\|_1= \sup_{u\in \calU(\calM)}\|\delta_a(u)\|_1.
	\end{align}
	Using this and \cref{prop:restrict-to-permutation-matrix} we find
	\begin{align*}
		\|\delta_a\|_{\infty,1} &= 
		\sup_{u\in \calU(\calM)}\|\delta_a(u)\|_{1} \\
		&= \sup_{u\in \calU(\Mat_m(\calN))}\|u^*[a,u]\|_{1}\\
		& =  \sup_{u\in \calU(\Mat_m(\calN))}\|u^*au - a\|_{1}\\
		&\leq  \sup_{u\in \calU(\Mat_m(\calN))}\|u^*au - a\|_{2}\\
		&=  \sup_{u\in \calU_m^{per}\otimes\textbf{1}_{\calN}}\|u^*au - a\|_{2}\\
		&= \sup_{\substack{T:\Omega_m\to \Omega_m \\ \text{permutation}}}\| g\circ T - g\|_2.
	\end{align*}
	The last step follows from the fact that, for  $u\in \calU_m^{per}\otimes\textbf{1}_{\calN}$, we have $u^*au = \Diag(g\circ T(1),\ldots, g\circ T(n))\otimes\textbf{1}_{\calN}$ for some permutation $T$.
	By \cref{lemma:optimality-triangle-function} we can fix a $g$ so that $\Diam(g(\Omega_m)) =1\leq  \widetilde{\Lambda}_m\inf_{z\in \CC}\|g-z\|_1$ (note that such $g$ is non-constant).
	Take any $b\in \calM$ with $\delta_a = \delta_b$. Then $a-b$ lies in the center of $\calM$, so $a-b=z_0\textbf{1}$ for some $z_0\in \CC$. Hence, $\|b\|_1 = \|a-z_0\textbf{1}\|_1 = \|g-z_0\|_1$ so that $\|\delta\|_{\infty,1}\leq \Diam(g(\Omega_m))\leq \widetilde{\Lambda}_m\|b\|_1$.
	The result now follows. Indeed, when $\calM$ is a I$_n$-factor, we obtained $\|\delta\|_{\infty,1}\leq \widetilde{\Lambda}_{n(\calM)}\|b\|_{1}$ and when $\calM$ is a II$_1$-factor we obtained $\|\delta\|_{\infty,1}\leq \widetilde{\Lambda}_{3}\|b\|_1 = \widetilde{\Lambda}_{\infty}\|b\|_1 = \widetilde{\Lambda}_{n(\calM)}\|b\|_1$.
\end{proof}
\end{theorem}

The following theorem shows that for (most) finite factors the distance formula from \eqref{bbs2} does not hold for arbitrary normal $a\in L_1(\calM,\tau)$, which shows a crucial difference with the classical result of Stampfli and its generalisations describing the norm of derivations $\delta_{a}:\calM\to \calM$, as for these derivations the distance formula \eqref{eq:derivation-distance-formula} holds for all $a\in \calM$. While the distance formula does not hold true, we are able to obtain constant bounds on the ratio  $\frac{\|\delta_{a}\|_{\infty,1}}{\min_{z\in \CC}\|a-z\mathbf{1}\|_1}$. In the case of II$_1$-factors and I$_n$-factors ($1<n<\infty$) with $n\equiv 0\mod 3$ these constants can not be improved.

\begin{theorem}\label{tderivationbound}
	Let $\calM$ be a finite factor with a faithful tracial state $\tau$ and let $a\in L_1(\calM,\tau)\setminus Z(\calM)$ be normal and measurable. Then the derivation $\delta_{a}:\calM\to L_1(\calM,\tau)$ satisfies:
	\begin{align}\label{eq:bounds-derivation}
		2\Lambda_{n(\calM)}\leq \frac{\|\delta_{a}\|_{\infty,1}}{\min_{z\in \CC}\|a-z\mathbf{1}\|_1} \leq 2.
	\end{align}
	Moreover, when $\calM\not=\CC$ there exist non-zero derivations $\delta_a,\delta_b$ corresponding to normal $a,b\in \calM$ such that $\|\delta_{a}\|_{\infty,1} \leq \widetilde{\Lambda}_{n(\calM)}\min_{z\in \CC}\|a-z\mathbf{1}\|_1$ and $\|\delta_{b}\|_{\infty,1} = 2\min_{z\in \CC}\|b-z\mathbf{1}\|_1$. We remark that
	\begin{enumerate}
		\item When $n(\calM)\not\in \{1,2,4\}$ then the distance formula of \eqref{bbs2} does not extend to arbitrary normal measurable $a\in L_1(\calM,\tau)\setminus Z(\calM)$, since $\widetilde{\Lambda}_{n(\calM)}<2$ in these cases.\label{remark-1}
		\item When $\calM$ is a II$_1$-factor or a $I_n$-factor with $n\equiv 0 \mod 3$ then the constant bounds given in \eqref{eq:bounds-derivation} can not be improved as in these cases $2\Lambda_{n(\calM)}=\sqrt{3} = \widetilde{\Lambda}_{n(\calM)}$.\label{remark-2}
	\end{enumerate}
	\begin{proof}
Let $a\in L_1(\calM,\tau)\setminus Z(\calM)$ be normal and measurable. By \cref{t3} there exist $u,w\in \calU(\calM)$, $z_0\in \CC$ satisfying \eqref{eq:optimal-constant-commutator} so that $\|\delta_{a}\|_{\infty,1} \geq \|\delta_{a}(u)\|_1 \geq 2\Lambda_{n(\calM)}\|a-z_0\mathbf{1}\|_1$, from which the first inequality follows.
The second inequality follows from the fact that $\|\delta_a(x)\|_1 = \|(a-z\mathbf{1})x - x(a-z\mathbf{1})\|_1\leq 2\|a-z\mathbf{1}\|_1\|x\|$ holds for any $x\in \calM$, $z\in \CC$.

For the next statement,  we note by \eqref{bbs2} that $\|\delta_{b}\|_{\infty,1}=2\inf_{z\in \CC}\|b-z\mathbf{1}\|_1$ holds for any self-adjoint $b\in \calM$, and that when $\calM\not=\CC$ we can choose $b$ so that moreover $b\not\in Z(\calM)$, ensuring that $\delta_b$ is non-zero. Moreover, by \cref{trel1}(2) we obtain a normal $a\in \calM$ such that $\delta_{a}$ is a non-zero derivation with $\|\delta_{a}\|_{\infty,1} \leq \widetilde{\Lambda}_{n(\calM)}\|a-z\mathbf{1}\|_1$ for every $z\in \CC$ since $\delta_{a} = \delta_{a-z\mathbf{1}}$. Thus $\|\delta_{a}\|_{\infty,1} \leq \widetilde{\Lambda}_{n(\calM)}\min_{z\in \CC}\|a-z\mathbf{1}\|_1$ (it is clear the minimum exists).  
The last two remarks follow directly.
\end{proof}
\end{theorem}

We remark that the above argument actually yields an estimate on the 
$L_1$-diameter of the unitary orbit 
$\calO(a)= \{uau^*: u\in \calU(\calM)\}$ of $a$. 
Indeed, as we already showed in  \eqref{eq:russo-dye-result}, we obtain by the Russo-Dye Theorem \cite[Theorem 1]{RuDy} that  $\|\delta_{a}\|_{\infty,1} = \sup_{u\in \calU(\calM)}\|\delta_a(u)\|_1$. Therefore
$$\Diam_{L_1(\calM,\tau)}(\calO(a)) 
= \sup_{u\in \calU(\calM)}\|a-uau^*\|_1 
= \sup_{u\in \calU(\calM)}\|\delta_a(u)\|_1 
= \|\delta_a\|_{\infty,1}.$$

\appendix
\section{}\label{section:appendix}
We prove two technical results concerning the constants $\Lambda_n$ and $\widetilde{\Lambda}_n$. In \cref{t_techmain} we will for $n\not=4$ determine the exact value of $\Lambda_n$ with the help of \cref{theorem-transformation-with-bound}. In \cref{lemma:optimality-triangle-function} we prove the main property of the constants $\widetilde{\Lambda}_n$ that we used in the paper.
 \begin{theorem}\label{t_techmain}
	We have $\Lambda_1=\Lambda_2=1$,
	$\frac{\sqrt{3}}{2}\leq \Lambda_4\leq 1$ and
	$\Lambda_n=\frac{\sqrt{3}}{2}$ for any $n\notin\{1,2,4\}$.
	
	Moreover, for $n\neq 4$ there exists a $g\in L_\infty(\Omega_n)$, $T\in \Aut_n, z\in\CC$ such that that $\Lambda(g,T,z)=\Lambda(g)=\Lambda_n$.
	
	\begin{proof}
		If $n=1$ then $\Lambda(g,\Id,g(1))=1$ for all $g\in\mathcal{S}(\Omega_n)$ since we agreed to count $\frac{0}{0}=1$. Hence, $\Lambda_1=1$.
		If $n=2$ then $\Lambda(g,T,\frac{g(1)+g(2)}{2})=1$ for all $g\in\mathcal{S}(\Omega_n)$ where $T(1)=2$. Hence, $\Lambda_2=1$.
		It follows from Theorem \ref{theorem-transformation-with-bound} that $\Lambda_n\geq \frac{\sqrt{3}}{2}$ for all $n\geq 3$. It only remains to show that this is in fact an equality whenever $n=3$ or $n\geq 5$, which we shall do now. For the given values of $n$, we can find a partition $\{A_1,A_2,A_3\}$ of $\Omega_n$ such that $\frac{1}{5}\leq \frac{\mu_n(A_j)}{\mu_n(\Omega_n)}\leq \frac{2}{5}$ for $j=1,2,3$. Now, denote  $w_j := e^{\frac{2\pi i j}{3}}$ for $j=1,2,3$ and construct the function
		$g = \sum_{j=1}^3 w_j\chi_{A_j}\in L_\infty(\Omega_n)$. We will show that $\Lambda(g)\leq
		\frac{\sqrt{3}}{2}$.\\
		
		Suppose $\Lambda(g)>\frac{\sqrt{3}}{2}$. Then there exists $T\in \Aut_n$, $z_0\in\CC$ and $\lambda>\frac{\sqrt{3}}{2}$ so that
		$$|g(T(\omega))-g(\omega)|\geq \lambda(|g(T(\omega))-z_0|+|g(\omega)-z_0|)$$
		a.e..
		
		We note that for $k\not=l$ we have
		$$|w_k - w_l| = \sqrt{3}.$$
		
		Denote $B_{k,j} = A_k\cap T^{-1}(A_j)$ so that $B_{k,j}\subseteq A_k$ and $T(B_{k,j})\subseteq A_j$. Moreover, since $\{A_1,A_2,A_3\}$ is a partition of $\Omega_n$, we have for $l=1,2,3$ that
		\begin{align}\label{eq:optimality-partition}
			A_l = B_{l,1}\cup B_{l,2}\cup B_{l,3} &\quad T^{-1}(A_l) = B_{1,l}\cup B_{2,l} \cup B_{3,l}.
		\end{align}
		We note that if $\mu_n(B_{k,j}\cup B_{j,k})>0$ we must by the assumption have that
		$$|w_k - w_j| \geq \lambda(|w_k - z_0| + |w_j-z_0|).$$
		This is to say that $z_0$ lies within the ellipse with foci $w_k$ and $w_j$ and eccentricity $\lambda$.

		Now suppose $\mu_n(B_{k,k})>0$ for some $k$. Then $z_0 = w_k$ and for $l,j\not=k$ we have
		$$|w_l - w_j| \leq \sqrt{3} < 2\lambda<2\lambda\sqrt{3}=\lambda(|w_l - w_k| + |w_j - w_k|)=\lambda(|w_l - z_0| + |w_j - z_0|)$$
		and hence $\mu_n(B_{l,j})=0$.
	However, \eqref{eq:optimality-partition} then implies for $j\not= k$  that $$\mu_n(A_j)= \mu_n(B_{j,1})+\mu_n(B_{j,2})+\mu_n(B_{j,3}) = \mu_n(B_{j,k}).$$ Therefore, using this and \eqref{eq:optimality-partition} we obtain
		\begin{align*}
			2\mu_n(A_{k}) &= \mu_n(A_k) + \left(\mu_n(B_{1,k})+\mu_n(B_{2,k})+\mu_n(B_{3,k})\right) \\
			&= \mu_n(A_k) + \left(\sum_{\substack{1\leq l\leq 3\\ l \not=k}}\mu_n(B_{l,k})\right)+\mu_n(B_{k,k}) \\
			&= \mu_n(A_k) + \left(\sum_{\substack{1\leq l\leq 3\\ l \not=k}}\mu_n(A_l)\right)+\mu_n(B_{k,k}) \\
			&= \mu_n(B_{k,k}) + \mu_n(A_1)+\mu_n(A_2)+\mu_n(A_3) \\
			&=\mu_n(\Omega_n) + \mu_n(B_{k,k})>\mu_n(\Omega_n).
		\end{align*}
		Hence $\frac{\mu_n(A_k)}{\mu_n(\Omega_n)}>\frac{1}{2}$, which is a contradiction with the choice of the partition.
		
		We conclude that $\mu_n(B_{k,k}) = 0$ for $k=1,2,3$. Now suppose that for some $1\leq l,j\leq 3$ with $l\not=j$ we have $\mu_n(B_{l,j}\cup B_{j,l}) = 0$. Let $k\in \{1,2,3\}$ such that $k\not=l,j$. Then we obtain $\mu_n(A_l)=\mu_n(B_{l,l})+\mu_n(B_{j,l})+\mu_n(B_{k,l}) = \mu_n(B_{k,l})$ and $\mu_n(A_j) = \mu_n(B_{l,j})+\mu_n(B_{j,j}) + \mu_n(B_{k,j}) = \mu_n(B_{k,j})$.
		We thus have
		\begin{align*}
			2\mu_n(A_{k}) &= \mu_n(A_{k}) + \mu_n(B_{k,l}) + \mu_n(B_{k,j}) +\mu_n(B_{k,k})\\
			&= \mu_n(A_k) + \mu_n(A_l)+\mu_n(A_j)= \mu_n(\Omega_n)
		\end{align*}
		and thus $\frac{\mu_n(A_k)}{\mu_n(\Omega_n)} =\frac{1}{2}$. This contradicts the choice of the partition sets.
		
		Hence, $\mu_n(B_{l,j}\cup B_{j,l})>0$ for all $l,j$ with $l\not=j$. This means that the point $z_0$ lies in all three ellipses (i.e. for $l\not=j$ the point $z_0$ has to lie inside the ellipse with foci $w_l$ and $w_j$ and eccentricity $\lambda$). We obtain that for $\lambda=\frac{\sqrt{3}}{2}$ the only point in the intersection of the three ellipses is $0$, and that for $\lambda>\frac{\sqrt{3}}{2}$ the intersection is empty (see \cref{figure:optimality})

		\begin{figure}[h!]
			\begin{tikzpicture}[baseline]
				\node at (0,0) 			(origin) {};
				\node at (0.3,-0.4)   		(O) {$O$};
				\node at (2,0)   			(v1) {$w_3$};
				\node at (-1,1.6)   		(v2) {$w_1$};
				\node at (-1, -1.6)   		(v3) {$w_2$};
				
				\draw[-,] (origin) -- node[right] {} (v1);
				\draw[-,] (origin) -- node[right] {} (v2);
				\draw[-,] (origin) -- node[right] {} (v3);
				
				\draw[-,] (v3) -- node[right] {} (v1);
				\draw[-,] (v1) -- node[right] {} (v2);
				\draw[-,] (v2) -- node[right] {} (v3);
				\filldraw (0,0) circle (2pt);
				\draw (-1,0) 		ellipse (28pt and 60pt);
				\draw[rotate=120] (-1,0) 		ellipse (28pt and 60pt);
				\draw[rotate=240] (-1,0) 		ellipse (28pt and 60pt);
				\draw (-1,0) 		ellipse (28pt and 60pt);
			\end{tikzpicture}
			\caption{The image of the simple function $g$ consists of the three points $w_1,w_2$ and $w_3$. The three ellipses with foci $w_l$ and $w_j$ (for $l$ and $j$ different) and eccentricity $\lambda = \frac{\sqrt{3}}{2}$ are drawn. The only point that lies in all three the ellipses is the point $z_0 :=0$.}
			\label{figure:optimality}
		\end{figure}
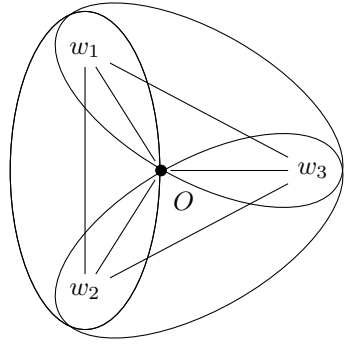
		Hence, $\Lambda(g)\leq\frac{\sqrt{3}}{2}$. Therefore $\Lambda_n=\frac{\sqrt{3}}{2}$.
	\end{proof}
	
\end{theorem}

\begin{lemma}\label{lemma:optimality-triangle-function}
	Let $1<n\leq \infty$. Then there is a $g\in L_\infty(\Omega_n)$ with $Diam(g(\Omega_n))=1$ and so that $\widetilde{\Lambda}_n=\sup_{z\in \CC}\frac{1}{\|g-z\|_1}$.
	\begin{proof} The result for $n=2$ follows directly by taking $g = \chi_{\{1\}}$.
		
		Thus, suppose $n\geq 3$. We can build a partition $\{A_1,A_2,A_3\}$ of $\Omega_n$ so that:
		\begin{itemize}
			\item  If $n=3k,\ k\in\NN$, or $n=\infty$, then $\mu_n(A_1)=\mu_n(A_2)=\mu_n(A_3) = \frac{1}{3}$.
			\item If $n=3k+1,\ k\in\NN$, then $\mu_n(A_1)=\mu_n(A_2)= \frac{k}{n},\ \mu_n(A_3) = \frac{k+1}{n}$.
			\item  If $n=3k+2,\ k\in\NN$, then $\mu_n(A_1)=\mu_n(A_2)= \frac{k+1}{n},\ \mu_n(A_3) = \frac{k}{n}$.
		\end{itemize}

		For convenience let us denote
		$$a=\mu_n(A_1)=\mu_n(A_2),\ b=\mu_n(A_3),\ w_k=e^{\frac{2\pi ki}{3}},\ k=0,1,2.$$
		
		Define $g_0\in L_\infty(\Omega_n,\mu_n)$ by $$g_0 = \chi_{A_1}w_1 + \chi_{A_2}w_2 + \chi_{A_3}w_0.$$
		Since $\mu_n(A_1)=\mu_n(A_2)$, it is clear that the minimum of $\CC\ni z\mapsto\|g_0 - z\|_1$ is attained for real-valued $z$, and moreover that $-\frac{1}{2}\leq z\leq 1$. When $n=4$, it is clear from the triangle inequality that the minimum is attained at the point $t_0=1$ and we have $\|g_0 -t_0\|_1 = \frac{\sqrt{3}}{2}$. Now assume $n\not=4$ so that the ratio $\frac{b}{a}$ satisfies $\frac{b}{a}< \sqrt{3}$ (the ratio $\frac{b}{a}$ is maximal for $n=7$ in which case we have $\frac{b}{a}=\frac{\frac{3}{7}}{\frac{2}{7}} = \frac{3}{2}<\sqrt{3}$). Hence $\sqrt{3}a-b>0$.  We have for $t\in [-\frac{1}{2},1]$ that
		$$\|g_0-t\|_1=2a|w_1-t|+b(1-t).$$
		Then
		$$\frac{d}{dt}\|g_0 - t\|_1=2a\frac{t+\frac{1}{2}}{|w_1-t|}-b.$$
		As $\frac{d}{dt}\|g_0 - t\|_1$ is negative when evaluated at $-\frac{1}{2}$ and positive when evaluated at $1$ (as $\sqrt{3}a-b>0$), the minimum of $\|g_0-t\|_1$ must be assumed at a point $t_0\in [-\frac{1}{2},1]$ satisfying
		$$b|w_1-t_0|=2a(t_0+\frac{1}{2}).$$
		Then
		$$b^2((t_0+\frac{1}{2})^2+\frac{3}{4})=4a^2(t_0+\frac{1}{2})^2$$
		and
		$$(t_0+\frac{1}{2})^2=\frac{3b^2}{4(4a^2-b^2)}=\frac{3b^2}{4(2a-b)}$$
		since $2a+b=1$. Therefore
		$$(t_0+\frac{1}{2})^2+\frac{3}{4}=\frac{3b^2}{4(2a-b)}+\frac{3}{4}=\frac{3a^2}{(2a-b)}$$ and
		\begin{align*}
			\|g_0-t_0\|_1 &= 2a|t_0-w_1| + b(1-t_0)\\
			&= 2\frac{\sqrt{3}a^2}{\sqrt{2a-b}}+b-b(\frac{\sqrt{3}b}{2\sqrt{2a-b}}-\frac{1}{2})\\
			&=\frac{\sqrt{3}\sqrt{2a-b}}{2}+\frac{3b}{2}\\
			&=\frac{\sqrt{3-6b}}{2}+\frac{3b}{2}.
		\end{align*}
		\begin{itemize}
			\item For $n=3k$ or $n=\infty$ we have $\mu_n(A_3)=\frac{1}{3}$ and find $\|g_0-t_0\|_1=1$.
			\item  For $n=3k+1$ ($n\not=4$) we have $\mu_n(A_3) = \frac{k+1}{3k+1}$ and find \[\|g_0 - t_0\|_1 = \frac{1}{2}\sqrt{\frac{3k-3}{3k+1}} + \frac{1}{2}\cdot\frac{3k+3}{3k+1}.\]
			\item For $n=3k+2$ we have $\mu_n(A_3) = \frac{k}{3k+2}$ and find \[\|g_0 - t_0\|_1 = \frac{1}{2}\sqrt{\frac{3k+6}{3k+2}} + \frac{1}{2}\cdot\frac{3k}{3k+2}.\]
		\end{itemize}
		
		Now, take $g = \frac{1}{\sqrt{3}}g_0$ so that $\Diam(g(\Omega_n))=1$. Then
		$$\sup_{z\in \CC}\frac{1}{\|g-z\|_1}=\sup_{z\in \CC}\frac{\sqrt{3}}{\|g_0-z\|_1}=\frac{\sqrt{3}}{\|g_0-t_0\|_1}=\tilde{\Lambda}_n.$$
	\end{proof}
\end{lemma}

\section*{Acknowledgements}
We like to thank Jinghao Huang and Thomas Scheckter for providing feedback that helped improve the exposition.

\bibliographystyle{amsalpha}

\end{document}